\documentclass[11pt]{amsart}

\usepackage{color}
\usepackage{xifthen}
\usepackage{amssymb}
\usepackage[table]{xcolor}
\usepackage{ytableau}
\usepackage{tikz}
\usepackage{varwidth}
\usepackage[all,cmtip]{xy}
\usetikzlibrary{calc, cd}
\usepackage[enableskew]{youngtab}
\usepackage{geometry}
\usepackage{enumerate}
\usepackage{array}
\usepackage[normalem]{ulem}
\usepackage{scalerel}
\usepackage{tikz-cd}
\usepackage[pagebackref]{hyperref}

\calclayout

\newtheorem{thm}{Theorem}[section]
\newtheorem{lemma}[thm]{Lemma}
\newtheorem{prop}[thm]{Proposition}

\newtheorem{cor}[thm]{Corollary}

\newtheorem{innercustomthm}{Theorem}
\newenvironment{customthm}[1]
  {\renewcommand\theinnercustomthm{#1}\innercustomthm}
  {\endinnercustomthm}

\theoremstyle{definition} 
\newtheorem{eg}[thm]{Example}

\newtheorem{defn}[thm]{Definition}

\newtheorem{obs}[thm]{Observation}

\newtheorem{rem}[thm]{Remark}
\newtheorem{qu}[thm]{Question}

\newcommand{\ds}{\displaystyle}
\newcommand{\eps}{\varepsilon}

\newcommand{\FF}{\mathbb{F}}

\newcommand{\PP}{\mathbb{P}}

\newcommand{\ZZ}{\mathbb{Z}}

\newcommand{\cC}{\mathcal{C}}
\newcommand{\cE}{\mathcal{E}}
\newcommand{\cG}{\mathcal{G}}

\newcommand{\cH}{\mathcal{H}}

\newcommand{\cL}{\mathcal{L}}
\newcommand{\cM}{\mathcal{M}}
\newcommand{\cO}{\mathcal{O}}

\newcommand{\cV}{\mathcal{V}}
\newcommand{\cW}{\mathcal{W}}

\newcommand{\on}{\operatorname}
\newcommand{\disp}{\on{disp}}

\DeclareMathOperator{\Aut}{Aut}

\DeclareMathOperator{\Pic}{Pic}

\newcommand{\iou}[1][]{
    \ifthenelse{\equal{#1}{}}{{\color{blue}\{IOU\}}}
    {{\color{blue}\{IOU: #1\}}}
}

\AtBeginDocument{%
   \def\MR#1{}
}

\title{Linear series with $\rho < 0$ via thrifty lego-building}
\author[N. Pflueger]{Nathan Pflueger}\address{Department of Mathematics and Statistics, Amherst College}\email{npflueger@amherst.edu}
\date{\today}

\newcommand{\debug}{0}
\renewcommand{\iou}[1][]{
 \ifthenelse{\equal{\debug}{1}}{
    \ifthenelse{\equal{#1}{}}{{\color{blue}\{IOU\}}}
    {{\color{blue}\{IOU: #1\}}}
  }{\empty}
}

\usetikzlibrary{patterns}

\newcommand{\tg}{\operatorname{tg}}
\newcommand{\geodiff}{\operatorname{g\delta}}
\newcommand{\ct}{\operatorname{ct}}
\newcommand{\combdiff}{\operatorname{c\delta}}
\newcommand{\floor}[1]{\left\lfloor #1 \right\rfloor}
\newcommand{\ceil}[1]{\left\lceil #1 \right\rceil}
\newcommand{\cpq}{C \backslash \{p,q\}}
\newcommand{\parenSkew}[2]{ \left( #1 \middle\slash #2 \right)}
\DeclareMathOperator{\wt}{wt}
\newcommand{\grdab}{G^{r, \alpha, \beta}_d}
\newcommand{\cgrdab}{\cG^{r,\alpha,\beta}_{g,d}}
\newcommand{\tgrdab}{\widetilde{G}^{r, \alpha, \beta}_d}

\newcommand{\dul}{\disp^+_\Lambda}
\newcommand{\ddl}{\disp^-_\Lambda}

\begin{document}
\maketitle

\begin{abstract}
The moduli space $\mathcal{G}^r_{g,d} \to \mathcal{M}_g$ parameterizing algebraic curves with a linear series of degree $d$ and rank $r$ has expected relative dimension $\rho = g - (r+1)(g-d+r)$. Classical Brill-Noether theory concerns the case $\rho \geq 0$; we consider the non-surjective case $\rho < 0$. We prove the existence of components of this moduli space with the expected relative dimension when $0 > \rho \geq -g+3$, or $0 > \rho \geq -C_r g + \mathcal{O}(g^{5/6})$, where $C_r$ is a constant depending on the rank of the linear series such that $C_r \to 3$ as $r \to \infty$. These results are proved via a two-marked-point generalization suitable for inductive arguments, and the regeneration theorem for limit linear series.
\end{abstract}

\iou[Disable ``debugging'' before submitting this manuscript.]

%%%%%%%%%%
%%%%%%%%%%
\section{Introduction}
\label{sec:intro}

Brill--Noether theory of algebraic curves can be understood, borrowing an analogy from \cite{harrisBNSurvey}, as ``representation theory for curves:'' given an abstract curve $C$ and integers $r,d$, how can the curve $C$ be mapped to $\PP^r$ by a degree $d$ map? 
The first question one might ask is: how plentiful are these maps? That is, what is the dimension of their parameter space? The central objects of Brill--Noether theory are two parameter spaces: $G^r_d(C)$, which parameterizes linear series $(\cL, V)$ such that $\deg \cL  = d$ and $\dim \PP V = r$, and $W^r_d(C)$, the image of $G^r_d(C) \to \Pic^d(C)$. Let $g$ be the genus of $C$. Given the numbers $g,r,d$, the \emph{Brill--Noether number} is
$$\rho = \rho(g,r,d) = g - (r+1)(g-d+r).$$
The \emph{Brill--Noether theorem} answers the ``how plentiful'' question when $C$ is a \emph{general} curve: $\dim G^r_d(C) = \rho$ unless $\rho < 0$, in which case $G^r_d(C)$ is empty. As one would hope, the same can be said of $W^r_d(C)$, except that when $\rho > g$, i.e. $g-d+r < 0$, $W^r_d(C)$ is all of $\Pic^d(C)$. This may also be phrased globally: these parameter spaces globalize to moduli spaces $\cG^r_{g,d} \to \cM_g$ and $\cW^r_{g,d} \to \cM_g$, and the Brill--Noether theorem says that $\cG^r_{g,d}$ has an irreducible component (in fact, a unique one, by another theorem) that surjects onto $\cM_g$ if and only if $\rho \geq 0$, and that component has relative dimension exactly $\rho$. This paper proves the following extensions to $\rho < 0$.

\begin{customthm}{A}
\label{thm:main}
Let $g,r,d$ be nonnegative integers, with $r+1, g-d+r \geq 2$. If 
$$\rho(g,r,d) \geq -g+3,$$
 then $\cG^r_{g,d}$ has a component of relative dimension $\rho$ and generic fiber dimension $\max\{0,\rho\}$ over $\cM_g$. 
 \end{customthm}
 
\begin{customthm}{B}
\label{thm:asy}
Define a sequence of rational numbers $(\eps_a)_{a \geq 1}$ by
$\ds \eps_a = \frac{ (-a) \operatorname{mod} 4}{\ceil{ a/4}},$
where $(-a) \operatorname{mod} 4$ denotes $4 \ceil{ \frac{a}{4}} - a  \in \{0,1,2,3\}$. 
For nonnegative integers $g,r,d$ with $d \leq g-1$, if
$$\rho(g,r,d) \geq (-3+\eps_{r+1}) g+ \cO(g^{5/6}),$$
then $\cG^r_{g,d}$ has a component of relative dimension $\rho$  and generic fiber dimension $\max\{0,\rho\}$ over $\cM_g$.
\end{customthm}

Lemma \ref{lem:asyPrecise} gives the explicit bound underlying Theorem \ref{thm:asy} without asymptotic notation.

We work over an algebraically closed field $\FF$. Except in Sections \ref{sec:ratlSeries} through \ref{sec:difficulty}, and in particular in the main theorems, we assume $\operatorname{char} \FF = 0$. In Sections \ref{sec:ratlSeries} through \ref{sec:difficulty}, statements assuming $\operatorname{char} \FF = 0$ are labeled. A ``curve'' is always assumed to be reduced, connected and proper. A ``scheme'' is always assumed to be finite-type over $\FF$, and by a ``point'' of a scheme we mean an $\FF$-point.

%%%%%%%%%%
%%%%%%%%%%
\subsection{Some remarks on the main theorems}
\label{ssec:thmRemarks}

In both theorems, when $\rho \leq 0$ the desired components are finite over $\cM_g$, from which it follows that a generic element is a \emph{complete} linear series and $\cG^r_{g,d} \to \cW^r_{g,d}$ is a local isomorphism (at least set-theoretically). So we can and will focus our attention on $\cG^r_{g,d}$ in this paper, except on a few occasions. A virtue of $\cW^r_{g,d}$ is that we have an isomorphism $\cW^r_{g,d} \xrightarrow{\sim} \cW^{g-d+r-1}_{g,2g-2-d}$ from Serre duality. This shows that we may swap the roles of $r+1$ and $g-d+r$ if we wish, and in particular assume without loss of generality that $r+1 \leq g-d+r$, i.e. $d \leq g-1$, in our arguments. In particular, that hypothesis $d \leq g-1$ in Theorem \ref{thm:asy} is harmless; it merely simplifies the constant in the bound. The bound $r+1, g-d+r \geq 2$ excludes the trivial case $r=0$ and the dual case $g-d+r -1 = 0$.

Note that $\eps_{r+1} \to 0$ as $r\to \infty$, and $\eps_{r+1} \leq 1$ for $r \geq 5$, so Theorem \ref{thm:asy} proves the existence of linear series with $r \geq 5, d \leq g-1$, and $\rho \geq -2g + o(g)$. For $r \equiv 3 \pmod{4}$, $\eps_{r+1} = 0$ and the bound is $\rho \geq -3g + o(g)$, which is asymptotically optimal since $\dim \cM_g = 3g-3$.
Evidently Theorem \ref{thm:asy} is much stronger than Theorem \ref{thm:main} for large $g$, but the error bound is formidable enough in low genus that Theorem \ref{thm:main} is stronger for $g \leq 1875$; see Remark \ref{rem:comparison}.

These moduli spaces $\cG^r_{g,d}, \cW^r_{g,d}$ are stacks, but the morphisms are representable by schemes. Concretely, this means that for any family $\cC \to S$ of smooth curves over a base scheme $S$, we obtain schemes $G^r_d(\cC) \to S$ and $W^r_d(\cC) \to S$, which we will call \emph{relative Brill--Noether schemes}, that are compatible with base change; see \cite[\S XXI]{acg2}.
We use the language of stacks mainly for linguistic convenience; everything we say can be formulated in scheme-theoretic terms via relative Brill--Noether schemes, and in fact we will be able to work almost exclusively set-theoretically since we are concerned with dimension statements.

The Brill--Noether number $\rho$ is a lower bound on the local relative dimension of $G^r_d(\cC) \to S$ at every point. Here, by \emph{local relative dimension} for a point $x$ in an $S$-scheme $f:G \to S$ we simply mean $\dim_x G - \dim_{f(x)} S$. We do not assume $f$ is surjective, so negative values are meaningful. Local relative dimension is preserved by smooth (in particular, \'etale) base change, so we can also define local relative dimension for $\cG^r_{g,d} \to \cM_g$ and $\cW^r_{g,d} \to \cM_g$ by choosing any versal deformation. By the relative dimension of an irreducible component, we mean the generic local relative dimension.

Local relative dimension can increase under base change from one smooth base scheme to another, but it cannot decrease. Any deformation $\cC \to S$ is \'etale-locally a pullback from a versal deformation, so if $L$ is a linear series at which the local relative dimension of $G^r_d(\cC) \to S$ is $\rho$ and $S$ is smooth, then the same is true for a \emph{versal} deformation, and therefore for $\cG^r_{g,d} \to \cM_g$. So in practice we will verify that $\cG^r_{g,d}$ has local relative dimension $\rho$ at a given point by checking that the relative local dimension is $\rho$ in a relative Brill--Noether scheme for a conveniently chosen not-necessarily-versal deformation over a smooth base. The same remark hold \emph{mutatis mutandis} when finding local relative dimension in $\cW^r_{g,d}$ or the other moduli spaces considered in this paper.

%%%%%%%%%%
%%%%%%%%%%
\subsection{Background}
\label{ssec:background}

Brill--Noether theory for $\rho < 0$, which I like to call ``underwater Brill--Noether theory'' (because the flora and fauna become more mysterious and hard to observe the deeper $\rho$ goes) encompasses several related questions.
Theorems \ref{thm:main} and \ref{thm:asy} address

\begin{qu}
\label{qu:underwaterBN}
For which $g,r,d$ such that $\rho < 0$ does $\cG^r_{g,d}$ (or equivalently, $\cW^r_{g,d}$) have an irreducible component of relative dimension exactly $\rho$ that is generically finite over $\cM_g$?
\end{qu}

Of course, one may ask a more demanding question.

\begin{qu}
\label{qu:underwaterBNUnique}
For which $g,r,d$ such that $\rho < 0$ is $\cG^r_{g,d}$ (or equivalently, $\cW^r_{g,d}$) \emph{equidimensional} of relative dimension $\rho$ and generically finite over $\cM_g$? When is it irreducible?
\end{qu}

Unfortunately, the methods of this paper are not applicable to irreducibility questions, since they are local in nature. Very little is known about Question \ref{qu:underwaterBNUnique}. It is known that $\cG^r_{g,d}$ is irreducible when $\rho = -1$ \cite{ehCodim1}, and equidimensional when $\rho=-2$ \cite{edidin93}.

Call a linear series \emph{very ample} if the map $C \to \PP V^\vee$ is an embedding. One may wish to consider only very ample linear series, and work in the Hilbert scheme of smooth, non-degenerate curves of degree $d$ and genus $g$ in $\PP^r$. Denote this by $\cH^r_{g,d}$; it has expected dimension (cf. \cite[ \S 1E]{hm})
$$h(g,r,d) = \rho(g,r,d) + \dim \cM_g + \dim \Aut \PP^r = (r+1)d - (r-3)(g-1).$$
Following \cite{pareschi89}, a component of $\cH^r_{g,d}$ is called \emph{regular} if $h^1(C,N_C) = 0$ for a general curve $C$ from it, and is said to have \emph{the expected number of moduli} if the image in $\cM_g$ has codimension $\max \{0, -\rho(g,r,d)\}$. A general curve in a regular component of $\cH^r_{g,d}$ is embedded by an isolated linear system if $\rho \leq 0$ \cite[Theorem 1.1.4]{pareschi89}, so if $\cH^r_{g,d}$ has a regular component with the expected number of moduli, then $\cG^r_{g,d}$ has a component of relative dimension $\rho$, generically finite over $\cM_g$.

\begin{qu}
\label{qu:underwaterHilb}
For which $g,r,d$ such that $\rho < 0$ does $\cH^r_{g,d}$ have a regular component with the expected number of moduli?
\end{qu}

%\begin{qu}
%\label{qu:underwaterHilbUnique}
%For which $g,r,d$ such that $\rho < 0$ is $\cH^r_{g,d}$ equidimensional of dimension $h(g,r,d)$? For which of these cases is it irreducible?
%\end{qu}

Our methods unfortunately do not provide new answers to Questions \ref{qu:underwaterHilb}; see Question \ref{qu:embed} and the discussion after it.

A nice, if dated, discussion of these and similar questions, with examples, can be found in \cite{montreal}; see also \cite[ \S 1E]{hm}. In 1982, Harris described the situation for $\rho < 0$ as ``truly uncharted waters'' \cite[p. 71]{montreal}, observing that there is such a wide range of observed behavior, and such a paucity of general principles, that the situation defies conjecture. However, one pattern that \emph{does} seem to emerge is that the wildest, most alien creatures in these mysterious waters live deep below the surface, where $\rho < -g +C$ for some constant $C$ (or perhaps some function $C(g) = o(g)$). 

Indeed, the simplest case where Question \ref{qu:underwaterHilb} has a negative answer is $\cH^3_{9,8}$. An octic curve of genus $9$ in $\PP^3$ is necessarily a complete intersection of a quadric and quartic surface, and an elementary dimension count shows that $\cH^3_{9,8}$ is irreducible with $\dim \cH^3_{9,8} = 33$; this is just larger than the expected $h(9,3,8) = 32$. In this case, $\rho = -7 = -g+2$. So at the depth $\rho = -g+2$ we encounter a slightly over-sized octicpus.
This does \emph{not immediately} imply that $\cG^3_{9,8}$ has no components of expected relative dimension, since there might be components consisting entirely of linear series giving maps to $\PP^3$ factoring through a lower-degree curve. Nonetheless, the fact that Theorem \ref{thm:main} reaches $\rho \geq -g+3$ and this exceptional behavior is found at $\rho = -g+2$ feels eerily compelling. Theorem \ref{thm:asy} reaches far beyond this depth; one could imagine that this means one of two things: either the dimensionally proper components it identifies live in their deep waters alongside wild and mysterious components of far larger dimension, or the threshold where the wild behavior occurs descends below $\rho \approx -g$ as $g$ increases.

\begin{rem}
The example $\cH^3_{9,8}$ above is a simple case of a Hilbert scheme of \emph{Castelnuovo curves}, which provide a large class of dimensionally improper linear series. For any $r \geq 2, d \geq 2r-1$, define $m = \lfloor \frac{d-1}{r-1} \rfloor$, $\eps = d-1 - m(r-1)$, and $g = \binom{m}{2} (r-1) + m \eps$. This is the maximum genus of a smooth, non-degenerate curve of degree $d$ in $\PP^r$, and such curves are called Castelnuovo curves. By the dimension count in e.g. \cite{ciliberto}, $\cH^r_{g,d}$ has a component of dimension $g + 2m + \eps + d-r-3 + \dim \Aut \PP^r$, whose image is a component in  $\cG^r_{g,d}$ of dimension $g + 2m + \eps + d -r - 3$. If we fix $r$ and let $d$ tend to infinity, we may write $m = \cO(\sqrt{g}), d = \cO(\sqrt{g})$, so this component has dimension $g + \cO(\sqrt{g})$, i.e relative dimension $-2g + \cO(\sqrt{g})$ over $\cM_g$, while $\rho(g,r,d) = -rg + \cO(\sqrt{g})$. So for $r \geq 3$, these components are dimensionally improper.
\end{rem}

\begin{rem}
A large collection of examples of non-dimensionally proper components of $\cG^r_{g,d}$ can also be gleaned from recent developments in Hurwitz--Brill--Noether theory. These developments have many connections to the present work, and suggest a Hurwitz space analog of Question \ref{qu:underwaterBN}; we summarize this in in Appendix \ref{app:hbn}.
\end{rem}

Several authors have charted more of these waters, particularly Question \ref{qu:underwaterHilb} (which then gives some answers to Question \ref{qu:underwaterBN} as well). 
The cases $r=1$ and $r=2$ are completely understood via the Hurwitz scheme and Severi variety, and results of Segre, Arbarello--Cornalba, Sernesi, and Harris.
In the case $r=3$, Pareschi \cite{pareschi89} answered Question \ref{qu:underwaterHilb} in an asymptotically optimal way; his result provides the ``thriftiest lego bricks'' we could hope for and are the essential input in our proof of the asymptotic Theorem \ref{thm:asy}. Strong results for $r \geq 4$ were later provided by Lopez \cite{lopez91,lopez99}. These are summarized in Section \ref{sec:previous}. 
Recently, Ballico \cite{ballico21} has proved the existence of regular components with the expected number of moduli in the range 
$$r \geq 4,\ d \geq r+1,\ (r+2)(d-r-1) \geq r(g-1).$$
Other results on similar lines include \cite{ballicoEllia88}.

Eisenbud and Harris asserted in \cite{eh86} that a forthcoming paper would prove $\cG^r_{g,d}$ has dimensionally proper points provided that $ \rho \geq \begin{cases}
-g+r+3 & \mbox{ if } r \equiv 1 \pmod{2},\\
-\frac{r}{r+2}g + r + 3 & \mbox{ if } r \equiv 0 \pmod{2}.
\end{cases}$
As far as I am aware, however, an argument was never published.

In an unpublished preprint \cite{negativeRho13}, I proved results on Question \ref{qu:underwaterBN} in a range somewhat smaller than what is proved in Theorem \ref{thm:main}. I must sheepishly admit that I never published it because I intended all these years to strengthen its results before resubmitting it. This intent has only now been realized by the present paper. I urge any graduate students or young researchers reading this paper to \emph{not follow this example} when preparing your thesis work for publication. Since that preprint has received some citations, I have decided to leave it ``in amber'' on the arXiv rather than updating it with the new content of this paper.

Work on similar problems to those discussed above includes \cite{farkasComponents} on regular components of moduli of stable maps to products of projective spaces, \cite{ballicoBenzoFontanari} on generalizations to nodal curves, and \cite{lopez18} on the problem of curves rigid in moduli.

%%%%%%%%%%
%%%%%%%%%%
\subsection{Dimensionally proper points and threshold genera}
\label{ssec:dpp}

We will see that Question \ref{qu:underwaterBN} is conveniently studied by obtaining bounds on \emph{threshold genera}, which we now define.

\begin{defn}
A linear series $L$ of degree $d$ and rank $r$ on a smooth curve $C$ of genus $g$ is called \emph{dimensionally proper} if $\cG^r_{g,d}$ has local relative dimension $\rho$ at $(C,L)$ and $G^r_d(C)$ has local dimension $\max \{0, \rho\}$ at $L$.
A line bundle $\cL$ on a curve $C$ is called \emph{dimensionally proper} if its complete linear series is dimensionally proper in the sense above.
\end{defn}

\begin{defn}
Let $a,b$ be positive integers. The \emph{threshold genus of the $a \times b$ rectangle}, denoted $\tg(a\times b)$, is the minimum genus $g_0$ such that, for all genera $g \geq g_0$, there exists a dimensionally proper line bundle $\cL$ on a genus $g$ curve $C$ such that $h^0(C,\cL) = a$ and $h^1(C,\cL) = b$.
Equivalently, $\tg(a \times b)$ is the minimum $g_0$ such that $\cG^{a-1}_{g,g + a - b - 1}$ has dimensionally proper points for all $g \geq g_0$.
\end{defn}

%We will obtain information Question \ref{qu:underwaterBN} using the implication
%\begin{equation}
%\label{eq:tgToG}
%g \geq \tg( (r+1) \times (g-d+r) ) \Rightarrow \cG^r_{g,d} \mbox{ has dimensionally proper points.}
%\end{equation}

The Brill--Noether theorem guarantees that $\tg(a \times b)$ is well defined and $\tg ( a \times b) \leq ab$. The function $\tg$ has a symmetry property  due to Serre duality:
\begin{equation}
\label{eq:tgSym}
\tg(a \times b) = \tg(b \times a).
\end{equation}
Results addressing Question \ref{qu:underwaterBN} are often conveniently formulated as bounds on $\tg(a\times b)$. Indeed, Theorems \ref{thm:main} and \ref{thm:asy}, respectively, will be proved by first proving the following inequalities.
\begin{eqnarray}
\label{eq:mainThmTG}
\tg(a \times b) &\leq& \frac12 ab + 2 \mbox{ for all } a,b \geq 2\\
\label{eq:asyThmTG}
\tg(a \times b) &\leq& \ceil{ \frac{a}{4} } b  + \cO( (ab)^{5/6}) \mbox{ for all } b \geq a  \geq 1
\end{eqnarray}
Therefore the rest of the paper will investigate Question \ref{qu:underwaterBN} via bounds on threshold genera $\tg(a \times b)$, and a flexible generalization of them to skew Young diagrams.

%%%%%%%%%%
%%%%%%%%%%
\subsection{Subadditivity and thrifty lego-building}
\label{ssec:strategy}

The engine driving this paper is a subadditivity theorem derived from the theory of limit linear series and proved in Section \ref{sec:subadd}. This subadditivity allows us to leverage existing results from many different sources to obtain new bounds. The first form of subadditivity concerns threshold genus of rectangles, where it may be stated as
\begin{equation}
\label{eq:tgSubadd}
\tg(a \times (b+c)) \leq \tg(a \times b) + \tg(a \times c).
\end{equation}
\begin{figure}

\begin{tikzpicture}[thick,scale=0.5]
\coordinate (p1) at (0,0);
\coordinate (p2) at (5,0);
\coordinate (q2) at (10,0);
\coordinate (p) at (18,0);
\coordinate (q) at (25,0);

\node [fill=black,circle, inner sep=2pt, label=above:$p_1$] at (p1) {};
\node [fill=black,circle, inner sep=2pt, label=above:{$q_1=p_2$}] at (p2) {};
\node [fill=black,circle, inner sep=2pt, label=above:{$q_2$}] at (q2) {};
\node [fill=black,circle, inner sep=2pt, label=above:$p$] at (p) {};
\node [fill=black,circle, inner sep=2pt, label=above:$q$] at (q) {};

\draw [shorten <= -0.4cm, shorten >= -0.4cm] (p1) to[out=-30, in=-150] node[midway,above] {$C_1$} (p2);
\draw [shorten <= -0.4cm, shorten >= -0.4cm] (p2) to[out=-30, in=-150] node[midway,above] {$C_2$} (q2);

% (2,3,4) / (0,0,2)
\draw (2.5,-2) -- (4.5,-2);
\draw (0.5,-3) -- (4.5,-3);
\draw (0.5,-4) -- (3.5,-4);
\draw (0.5,-5) -- (2.5,-5);
\draw (0.5,-3) -- (0.5,-5);
\draw (1.5,-3) -- (1.5,-5);
\draw (2.5,-2) -- (2.5,-5);
\draw (3.5, -2) -- (3.5,-4);
\draw (4.5,-2) -- (4.5,-3);

% (5,7,7) / (2,3,4)
\draw (7.5,-2) -- (10.5,-2);
\draw (6.5,-3) -- (10.5,-3);
\draw (5.5, -4) -- (10.5,-4);
\draw (5.5,-5) -- (8.5,-5);
\draw (5.5,-4) -- (5.5,-5);
\draw (6.5, -3) -- (6.5,-5);
\draw (7.5,-2) -- (7.5,-5);
\draw (8.5,-2) -- (8.5,-5);
\draw (9.5,-2) -- (9.5,-4);
\draw (10.5,-2) -- (10.5,-4);

% (5,7,7) / (0,0,2)
\begin{scope}[shift = {(18.5,-2)}]
\draw (2,0) -- (7,0);
\draw (0,-1) -- (7,-1);
\draw (0,-2) -- (7,-2);
\draw (0,-3) -- (5,-3);
\draw (0,-1) -- (0,-3);
\draw (1,-1) -- (1,-3);
\draw (2,0) -- (2,-3);
\draw (3,0) -- (3,-3);
\draw (4,0) -- (4,-3);
\draw (5,0) -- (5,-3);
\draw (6,0) -- (6,-2);
\draw (7,0) -- (7,-2);
\draw[ultra thick] (4,0) -- (4,-1) -- (3,-1) -- (3,-2) -- (2,-2) -- (2,-3);
\end{scope}

\draw[->] (12,0) -- (16,0) node[midway, above]{regenerate};

\draw [shorten <= -0.5cm, shorten >= -0.5cm] (p) to[in=-150, out=-30] node[midway, above] {$C$} (q);

\end{tikzpicture}

\caption{An illustration of the proof of the subadditivity theorem. A point of $\widetilde{\cG}^{\beta / \alpha}_{g_1}$ and a point of $\widetilde{\cG}^{\gamma / \beta}_{g_2}$ are used to construct a limit linear series on a nodal curve of genus $g_1+g_2$. This curve is smoothed to obtain a point of $\widetilde{\cG}^{\gamma / \alpha}_{g_1+g_2}$.
}
\label{fig:subadd}
\end{figure}
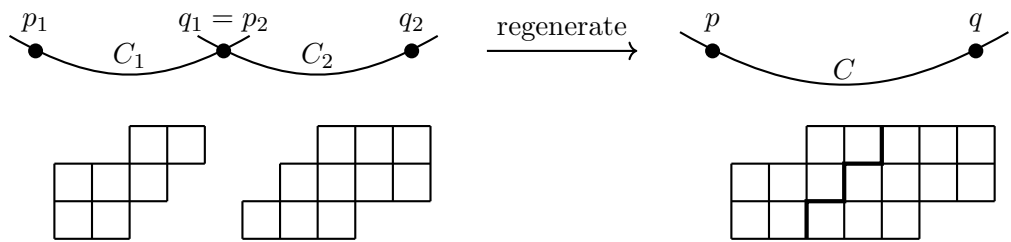
More generally, the notion of threshold genera generalizes in a natural way to \emph{fixed-height skew shapes} $\beta / \alpha$, as defined in Section \ref{sec:ratlSeries}, and we prove subadditivity in this more general context, as an inequality $\tg( \gamma / \alpha) \leq \tg(\gamma / \beta) + \tg(\beta / \alpha)$ (Theorem \ref{thm:subadd}). 
An illustration of the proof of subadditivity is shown in Figure \ref{fig:subadd}; the notation in the caption and the proof itself are in Section \ref{sec:subadd}.
The reader may understand this via a cheesy analogy: the threshold genus of a fixed-height skew shape tells the cost of a building an object constructed out of little square bricks (the boxes of the skew Young diagram), and subadditivity indicates that you can assemble this object from multiple pieces, paying for the parts in each piece separately. Our goal, then, in proving Theorems \ref{thm:main} and \ref{thm:asy}, is to identify particularly inexpensive ways to construct rectangles; we refer to this process as ``thrifty lego-building.'' See e.g. Figure \ref{fig:fourStage} for the way we thriftily build rectangles for Theorem \ref{thm:main}.

In the proofs of both theorems, we will use existing results to obtain bounds on certain smaller skew shapes that are assembled into rectangles. To stretch the lego analogy a bit further, we visit the toy store and find good deals of a few boxed sets that will then be put together to cheaply construct rectangles. The proof of Theorem \ref{thm:main} involved much more intricate lego-building, using results of Komeda \cite{kom91} on Weierstrass points and a careful analysis of skew shapes of threshold genus $1$. The proof of Theorem \ref{thm:asy}, by contrast, uses only a very simple sort of lego building: many thin rectangles are stacked up to a thicker one. The great deal used in that proof is provided by Pareschi \cite{pareschi89}, which gives an asymptotically optimal bound of $\tg(4 \times b)$. The disadvantage of this much simpler lego-building is of course the somewhat complicated error bounds.

%%%%%%%%%%
%%%%%%%%%%
\subsection{Some further questions}
\label{ssec:furtherQs}

An intriguing asymptotic version of Question \ref{qu:underwaterBN} is
\begin{qu}
\label{qu:limit}
For fixed $a$, what is $\ds \lim_{b \to \infty} \frac{\tg(a \times b)}{b}$?
\end{qu}
Subadditivity of threshold genus implies that the limit exists, and is equal to $\ds \inf_{b \geq 1} \frac{\tg(a \times b)}{b}$. Equation \eqref{eq:asyThmTG} and the fact that $\rho(g,r,d) \geq - \dim \cM_g$ whenever $\cG^r_{g,d}$ has dimensionally proper components imply bounds
$\ds \frac{a}{4} \leq \lim_{b \to \infty} \frac{\tg(a \times b)}{b} \leq \ceil{ \frac{a}{4} }.$

Conjecturally, the bound $\rho \geq -\dim \cM_g$ never holds with equality for $g > 0$; this is a folklore conjecture sometimes called the \emph{rigid curves conjecture}, and it has recently been verified for $r=3$ and in a restricted range of cases with $r \geq 4$ by Keem, Kim, and Lopez \cite{lopez18}. A natural question, in light of this conjecture, is: how close can you get? In the language of threshold genus,
\begin{qu}
What is $\ds \inf_{a,b > 0} \left( \tg(a \times b) - \frac14 ab \right)$?
\end{qu}

Finally, we note that one unsatisfying aspect of our main results is that it is not clear whether the components produced consist of very ample linear series.

\begin{qu}
\label{qu:embed}
Do the irreducible components produced in the proofs of Theorems \ref{thm:main} and \ref{thm:asy} correspond to regular components of $\cH^r_{g,d}$ with the expected number of moduli? In other words, do they provide answers to Question \ref{qu:underwaterHilb} as well?
\end{qu}

Our method does not provide information about this question, because our subadditivity theorem does not say anything about whether the resulting linear series are very ample.
It is possible that the subadditivity theorem can be strengthened to include this information after adding some additional hypotheses, but new ideas are required.

%%%%%%%%%%
%%%%%%%%%%
\subsection{A note on the characteristic $0$ hypothesis}
\label{ssec:char0}

The hypothesis $\operatorname{char} \FF = 0$ is used in two essential ways. The first is that we build on results with that hypothesis: the proof of Theorem \ref{thm:main} uses the results of \cite{kom91}, and the proof of Theorem \ref{thm:asy} uses results of \cite{pareschi89}, both of which assume $\operatorname{char} \FF = 0$. I do not know for certain if those results, or at least the parts needed for our application, can be extended to characteristic $p$. Secondly, we use the fact that, in characteristic $0$, any linear series has finitely many ramification points (see e.g. \cite[p. 39]{acgh}).

I strongly suspect that very slightly weakened version of Theorem \ref{thm:main} can be proved in characteristic $p$ by a similar strategy, replacing the use of Komeda's theorem with a different bound valid in all characteristics.
The proof of Theorem \ref{thm:asy}, however, cannot easily be extended to arbitrary characteristic, since it depends indispensably on \cite{pareschi89}.

Sections \ref{sec:ratlSeries} through \ref{sec:difficulty} do not assume that $\operatorname{char}\FF = 0$ except where stated otherwise, so that the results therein can be used in any future work addressing the situation in characteristic $p$. For such applications, it is probably necessary to modify the definition of $\tg(a \times b)$ to consider only the open locus of linear series that have unramified points. Unfortunately, symmetry of $\tg(a \times b)$ in $a$ and $b$ is lost, since this property need not be preserved under Serre duality. Our notion of threshold genus of skew shapes, however, requires no modification in positive characteristic.

%%%%%%%%%%
%%%%%%%%%%
\section{Threshold genera for $a \leq 4$}
\label{sec:previous}

This section surveys some previously known results about Question \ref{qu:underwaterBN}, mostly via work on Question \ref{qu:underwaterHilb}, formulated as bounds on threshold genera. In light of the symmetry in $a,b$, we will assume $a \leq b$ in this discussion. We begin by surveying results from the literature that show the following asymptotic facts, explained in the next several subsections.
\begin{equation}
\label{eq:a3asy}
\tg(1 \times b) = b \hspace{1cm}
\tg(2 \times b) = b + 1 \hspace{1cm}
\tg(3 \times b) = b + \cO(\sqrt{b}) \hspace{1cm}
\tg(4 \times b) = b + \cO(b^{2/3})
\end{equation}
In particular, this confirms that the first four answers to Question \ref{qu:limit} are ``$1$.''

%%%%%%%%%%
%%%%%%%%%%
\subsection{The case $a = 1$}
\label{ssec:a1}

\begin{lemma}
\label{lem:gdMax}
For all $a,b \geq 1$, $\tg(a \times b) \geq a+b-1$.
\end{lemma}
\begin{proof}
Suppose $\tg(a \times b) \leq g \leq ab$, and $\cL$ is a dimensionally proper line bundle on a genus-$g$ curve $C$ with $h^0(C,\cL) = a$ and $h^1(C,\cL) = b$. Clifford's theorem and Riemann-Roch imply $a-1 \leq \frac12 \deg \cL = \frac12(g-1+a-b)$, which simplifies to $g \geq a+b-1$.
\end{proof}

Therefore $a+b-1 \leq \tg(a \times b) \leq ab$ for all $a,b>0$; in particular 
\begin{equation}
\label{eq:a1}
\tg(1 \times b) = b \mbox{ for all $b \geq 1$.}
\end{equation}

%%%%%%%%%%
%%%%%%%%%%
\subsection{The case $a=2$}
\label{ssec:a2}
The case $a = 2$ concerns the geometry of $\cG^1_{g,d}$. This case can be analyzed via the Hurwitz space $\cH_{d,g}$ of degree-$d$ branched covers $f: C \to \PP^1$ from a genus $g$ curve, which can be identified (set-theoretically) with the open subspace in $\cG^1_{g,d}$ of basepoint-free series.
The forgetful map $\cH_{d,g} \to \cM_g$ is generically finite; this was proved by Segre \cite{segre28}, and a modern treatment and strengthening is in \cite{acFootnotes}. The space $\cH_{d,g}$ is nonempty if and only if $d \geq 2$, which is equivalent to $\rho(g,1,d) \geq -g + 2$, and if so it is irreducible of dimension $2g+2d-5 = \rho(g,1,d) + \dim \cM_g$. So the closure of the basepoint-free locus gives a dimensionally proper component of $\cG^1_{g,d}$ provided that $\rho(g,1,d) \geq -g+2$. Thus $\tg(2 \times b) \leq b+1$. Lemma \ref{lem:gdMax} gives the reverse inequality, hence
\begin{equation}
\label{eq:a2}
\tg(2 \times b) = b+1 \mbox{ for all $b \geq 1$.}
\end{equation}

%%%%%%%%%%
%%%%%%%%%%
\subsection{The case $a=3$}
\label{ssec:a3}

The case $a=3$ concerns $\cG^2_{g,d}$, which may be studied using the geometry of Severi varieties, and in particular using a theorem of Sernesi \cite{sernesiExistence}. We continue to assume $a \leq b$, i.e. $b \geq 3$. Sernesi considered the Severi variety $\cV_{d,g}$ of irreducible nodal plane curves of degree $d$ and geometric genus $g$, and proved that if
\begin{equation}
\label{eq:sernesiBounds}
d \geq 5 \mbox{ and } d-2 \leq g \leq \binom{d-1}{2},
\end{equation}
then $\cV_{d,g}$ has an irreducible component%
\footnote{Later, Harris proved that in fact $\cV_{d,g}$ is irreducible  \cite{harrisSeveri}, so there is only one component to speak of.}
whose image in $\cM_g$ has codimension $\max\{ 0, -\rho(g,2,d) \}$. Up to the action of $\operatorname{PGL}_3$, $\cV_{d,g}$ can be identified with the open locus in $\cG^2_{g,d}$ of linear series giving immersions in $\PP^2$ with nodal images, so this shows that the closure of this locus is an irreducible component of $\cG^2_{g,d}$, of dimension $\dim \cV_{d,g} - \dim \operatorname{PGL}_3 = \dim \cM_g + \rho(g,2,d)$. Therefore this is a dimensionally proper component of $\cG^2_{g,d}$ when $d \geq 5$ and $d-2 \leq g \leq \binom{d-1}{2}$.

Translated into our terms, suppose that we have fixed $a = 3 \leq b$, and let $r = 2, d = g+a-b-1 = g+2-b$. We want to know the minimum $g$ such that the bounds $g \geq b+3$ and $g-b \leq g \leq \binom{g+1-b}{2}$ hold. By elementary algebra, that minimum is $\ceil{b + \frac12 + \sqrt{2b+\frac14}}$, hence
\begin{equation}
\label{eq:a3}
\tg(3 \times b) \leq b + \ceil{ \frac12 + \sqrt{2b + \frac14}\ } \mbox{ for all } b \geq 3.
\end{equation}
Further algebra shows that $\frac12 + \sqrt{2b + \frac14} \leq \frac12 b + \frac32$ for all $b \geq 3$; equality holds for $b=3$. Since $\ceil{\frac12 b + \frac32} \leq \frac12 b + 2$, Inequality \eqref{eq:a3} implies Inequality \eqref{eq:mainThmTG} in the case $a=3$:
\begin{equation}
\label{eq:maina3}
\tg(3 \times b) \leq \frac32 b + 2.
\end{equation}

%%%%%%%%%%
%%%%%%%%%%
\subsection{The case $a=4$}
\label{ssec:a4}
In the case $a=4$, an asymptotically precise description of $\tg(4 \times b)$ follows from work of Pareschi \cite{pareschi89}. Pareschi considers the Hilbert scheme $\cH^3_{g,d}$, i.e. it addresses Question \ref{qu:underwaterHilb}.
Pareschi's theorem says that $\cH^3_{g,d}$ has a regular component with the expected number of moduli provided that
$d \geq 20,\ \rho(g,3,d) \leq 0,\ \mbox{ and } f(d) \geq g,$ where $f(d)$ is the following function.
\begin{eqnarray*}
f(d) &=& \frac{1}{1536} (32d-215)^{3/2} + \frac{23}{16}d - \frac{541}{1536} (32d-215)^{1/2} - \frac{397}{128}\\
&=& \frac{1}{6 \sqrt{2}} d^{3/2} + \cO( d )
\end{eqnarray*}
Observe that $f(d)$ is increasing for $d \geq 20$; it follows that 
\begin{eqnarray}
\tg(4 \times b) &\leq& \min \left\{ 
g:\ g+3-b \geq 20 \mbox{ and } f(g+3-b) \geq g
\right\}\\
\label{eq:tg4bexplicit}
&=& b-3 + \min \{ d \geq 20: f(d) -d + 3 \geq b \}
\end{eqnarray}
and therefore
\begin{equation}
\label{eq:a4}
\tg(4 \times b) = b + \cO(b^{2/3}).
\end{equation}

\begin{rem}
Our bound from Theorem \ref{thm:main} for this case is $\tg(4 \times b) \leq 2b + 2$, which is obviously much weaker than Equation \eqref{eq:a4} for large $b$. Nonetheless, Theorem \ref{thm:main} is still stronger in reasonably high genus: the explicit bound in Equation \eqref{eq:tg4bexplicit} is greater than $2b+2$ for all $b \leq 50$, and is equal to $2b+2$ for $51 \leq b \leq 53$. So even in genus $g=102$, Theorem \ref{thm:main} gives a stronger bound on $\tg(4 \times b)$.
\end{rem}

%%%%%%%%%%
%%%%%%%%%%
\subsection{The cases $a \geq 5$}

In the cases $a \geq 5$, i.e. $r \geq 4$, strong asymptotic results about Question \ref{qu:underwaterHilb} have been obtained by Lopez \cite{lopez99}, building on \cite{lopez91}. Some of these results are summarized in the table below in asymptotic form; see \cite{lopez99} for precise bounds.

\begin{center}
$\begin{array}{rll}
\mbox{rank} & \mbox{ sufficient bound on $\rho$} & \mbox{bound on threshold genus}\\
 r=4 & -\frac32 g + \cO(1) \leq \rho \leq 0 & \tg(5 \times b) \leq 2b + \cO(1)\\[4pt]
 r=5 &  -\frac54 g + \cO(1) \leq \rho \leq 0 &  \tg(6 \times b) \leq \frac{8}{3} b + \cO(1)\\[4pt]
r=6 & -\frac{30}{19} g + \cO(1) \leq \rho \leq 0 & tg(7 \times b) \leq \frac{19}{7} b  + \cO(1)\\[4pt]
r\geq7 &  -(2- \frac{6}{r+3})g - \cO(r^2) \leq \rho \leq 0 &  \tg( a \times b) \leq \frac{a+2}{3} b + \cO(a^2) \mbox{ for $a \geq 8$.}
\end{array}$
\end{center}

Note in particular that Lopez's results in $r=4$ gives the same asymptotic as Theorem \ref{thm:asy}, while a stronger asymptotic is obtained in Theorem \ref{thm:asy} for $r \geq 5$.  Like the work of Pareschi, these results concern regular components of the Hilbert scheme with the expected number of moduli, so they are stronger than merely bounds on $\tg(a \times b)$, and remain the strongest results on Question \ref{qu:underwaterHilb}.

%%%%%%%%%%
%%%%%%%%%%
\section{Rational linear series on twice-marked curves}
\label{sec:ratlSeries}

Our main theorems, though stated for smooth curves without marked points, will be proved by inductive arguments about curves with two marked points that are chained together and smoothed. This section develops preliminary material about Brill--Noether theory of curves with two marked points. This section contains no original mathematics, but develops some slightly nonstandard terminology about ``rational linear series'' that will prove to be very convenient for our purpose, and formulates some known results in these terms. In this section, and every section up to and including Section \ref{sec:difficulty}, we \emph{do not assume that $\operatorname{char} \FF = 0$} unless stated otherwise.

%%%%%%%%%%
%%%%%%%%%%
\subsection{Ramification sequences and fixed-height skew shapes}
\label{ssec:ramSeqs}

By a \emph{rank-$r$ ramification sequence} we mean an $(r+1)$-tuple $\alpha = (\alpha_0, \cdots, \alpha_r)$, where $\alpha_0 \leq \cdots \leq \alpha_r$. We do not assume that all $\alpha_i \geq 0$. Instead, we call a ramification sequence \emph{nonnegative} if this is so. Ramification sequences will always be denoted by lowercase greek letters.
For any integer $n$ and ramification sequence $\alpha$, we write $n + \alpha$ for $(n+\alpha_0, \cdots, n+ \alpha_r)$ and $n - \alpha$ for $(n-\alpha_r, \cdots, n-\alpha_0)$.
Denote the sum $\sum_{i=0}^r \alpha_i$ of a ramification sequence by $|\alpha|$, and $|\beta / \alpha|$ denotes $|\beta| - |\alpha|$. We will sometimes use exponential notation for ramification sequences. For example, $0^2 1^3 2^1$ denotes $(0,0,1,1,1,2)$, and $(g-d+r)^{r+1}$ denotes $(g-d+r, g-d+r, \cdots, g-d+r)$, where the sequence has rank $r$.

If $\alpha$ is a ramification sequence, the \emph{associated vanishing sequence} is the strictly increasing sequence $a = (a_0, \cdots, a_r)$ where $a_i = \alpha_i + i$. We follow the convention that a vanishing sequence is denoted by the lowercase roman letter corresponding greek letter of the ramification sequence.

We write $\alpha \leq \beta$ to mean that $\alpha,\beta$ have the same rank $r$ and $\alpha_i \leq \beta_i$ for all $0 \leq i \leq r$. The notation $\beta / \alpha$ denotes an ordered pair of two ramification sequences, in which we assume $\alpha \leq \beta$. The ordered pair $\beta / \alpha$ is usefully visualized as a skew Young diagram, as illustrated in Figure \ref{fig:skewShape}, though this visualization is not logically necessary in our discussion.
For this reason we will call such an ordered pair a \emph{fixed-height skew shape of rank $r$}. We emphasize the phrase ``fixed-height'' here: $r$ is part of the data of $\beta / \alpha$, as opposed to the standard definition of a skew shape as an ordered pair of partitions; see Remark \ref{rem:fixedHeight}.

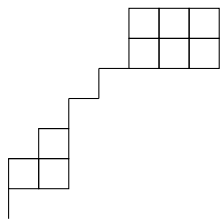
\begin{figure}
\begin{tikzpicture}
\draw (0.000, 0.000) -- (0.000, 0.400);
\draw (0.000,0.400) rectangle ++(0.400,0.400);
\draw (0.400,0.400) rectangle ++(0.400,0.400);
\draw (0.000,0.800) -- (0.400, 0.800);
\draw (0.400,0.800) rectangle ++(0.400,0.400);
\draw (0.400,1.200) -- (0.800, 1.200);
\draw (0.800, 1.200) -- (0.800, 1.600);
\draw (0.800,1.600) -- (1.200, 1.600);
\draw (1.200, 1.600) -- (1.200, 2.000);
\draw (1.200,2.000) -- (1.600, 2.000);
\draw (1.600,2.000) rectangle ++(0.400,0.400);
\draw (2.000,2.000) rectangle ++(0.400,0.400);
\draw (2.400,2.000) rectangle ++(0.400,0.400);
\draw (1.600,2.400) rectangle ++(0.400,0.400);
\draw (2.000,2.400) rectangle ++(0.400,0.400);
\draw (2.400,2.400) rectangle ++(0.400,0.400);
\end{tikzpicture}
\caption{
The Young diagram of the fixed-height skew shape $(0,2,2,2,3,7,7) / (0,0,1,2,3,4,4)$.
}
\label{fig:skewShape}
\end{figure}

%%%%%%%%%%
%%%%%%%%%%
\subsection{Ramification of linear series}
\label{ssec:ram}

Let $C$ be a smooth curve of genus $g$, and let $L = (\cL,V)$ be a linear series of rank $r$ and degree $d$ on $(C,p,q)$. For any $p \in C$, the \emph{vanishing sequence} of $L$ at $p$ is the increasing sequence $a(L,p) = (a_0(L,p), \cdots, a_r(L,p))$ of orders of vanishing of sections of $V$ at $p$. The \emph{ramification sequence} of $L$ at $p$ is $\alpha(L,p)$, where
$\alpha_i(L,p) = a_i(L,p) - i.$
If $| \alpha(L,p) | > 0$ we call $p$ a \emph{ramification point} of $L$. Otherwise, it is called \emph{unramified}.

%%%%%%%%%%
%%%%%%%%%%
\subsection{Brill--Noether theory of twice-marked curves}
\label{bnTwiceMarked}
Let $(C,p,q)$ be a smooth twice-marked curve, fix integers $r,d \geq 0$, and let $\alpha, \beta$ be nonnegative ramification sequences of rank $r$. Define 
$$G^{r,\alpha,\beta}_d(C,p,q) = \{[L] \in G^r_d(C):\ \alpha(L, p) \geq \alpha,\ \alpha(L,q) \geq \beta \}.$$
Denote by $\tgrdab(C,p,q)$ the open subscheme where $\alpha(L,p) = \alpha$ and $\alpha(L,q) = \beta$.
This construction globalizes to a morphism of stacks $\cgrdab \to \cM_{g,2}$, representable by schemes. For a family $f: \cC \to S$ with two disjoint sections $p_S, q_S$, this gives an $S$-scheme $\grdab(\cC,p_S,q_S)$, which we call a \emph{relative Brill--Noether scheme} for twice-marked curves.
By imposing each ramification condition locally via the inverse image in $\cG^r_{g,d}$ of a Schubert variety, it follows that the local relative dimension of $\cgrdab \to \cM_{g,2}$ is bounded locally at all points by the \emph{adjusted Brill--
Noether number}:
$$\rho = \rho(g,r,d,\alpha, \beta) = \rho(g,r,d) - |\alpha| - |\beta|.$$

\begin{defn}
A linear series $L$ on a twice-marked curve $(C,p,q)$, with ramification at least $\alpha$ at $p$ and $\beta$ at $q$, is called \emph{dimensionally proper in $\cgrdab$} if the local dimension of $\grdab(C,p,q)$ is $\max\{0, \rho\}$ and the local relative dimension of $\cgrdab \to \cM_{g,2}$ is $\rho$.
%The remarks in Section \ref{ssec:thmRemarks} apply here: the local relative dimension of $\cgrdab \to \cM_{g,2}$ is $\rho$ if and only if there is \emph{some} deformation $(\cC_S, p_S, q_S)$ of $(C,p,q)$ over a smooth base scheme $S$ for which the local relative dimension of $\grdab(\cC_S, p_S, q_S)$ at $L$ is $\rho$.
\end{defn}

%The analog of the Brill--Noether theorem for twice-marked curves is the following.

\begin{thm}[Brill-Noether theorem for two marked points; see \cite{eh86} or \cite{ossSimpleBN}]
\label{thm:twoPointBN}
Suppose $g \geq 0$, and let $(C,p,q)$ be a general twice-marked smooth curve of genus $g$. For nonnegative ramification sequences $\alpha, \beta$ of rank $r$, $\grdab(C,p,q)$ is nonempty if and only if the number
$$\hat{\rho} = \hat{\rho}(g,r,d,\alpha,\beta) = g- \sum_{i=0}^r \max \left\{ 
0,
g-d+r + \alpha_{r-i} + \beta_{i} \right\}$$
is nonnegative. If it is nonempty, then it has pure dimension $\rho$.
\end{thm}

\begin{rem}
The number $\hat{\rho}$ has a geometric interpretation: it is the expected dimension not of $G^{r,\alpha,\beta}_d(C,p,q)$ itself, but of its image in $\Pic^d(C)$. %Indices such that $\alpha_{r-i} + \beta_i \leq d - g - r$ reflect freedom to vary the choice of vector space of sections after choosing the line bundle.
\end{rem}

%%%%%%%%%%
%%%%%%%%%%
\subsection{Extending to ramification sequences with negative entries}
Several arguments in this paper are simplified if we relax the assumption that $\alpha, \beta$ are nonnegative in the discussion above, and allow ``linear series'' whose sections have poles of bounded order at the marked points.

\begin{defn}
A \emph{rational linear series} of rank $r$ and degree $d$ on $(C,p,q)$ is a pair $L = (\cL, V)$ of a degree $d$ line bundle on $C$ and an $(r+1)$-dimensional vector space $V \subseteq H^0(\cpq, \cL)$ of \emph{rational} sections of $\cL$ that are regular except possibly at $p,q$. So $L$ encodes an $r$-dimensional family of divisors that are effective away from $p$ and $q$, but may have negative multiplicity at $p$ or $q$.

If $L$ is a rational linear series on $(C,p,q)$, define ramification sequences $\alpha(L,p)$ and $\alpha(L,q)$ in the same way as before, regarding vanishing order of a section with a pole to be negative.
\end{defn}

If $L$ is a rational linear series on $(C,p,q)$ of degree $d$, then for any $m,n \in \ZZ$ we may define a rational linear series $L + mp + nq$ of degree $d+m+n$, given by twisting by $\cO_C(mp+nq)$. For $m,n$ sufficiently large we obtain a \emph{regular} linear series, with ramification $\alpha(L,p)+m$ at $p$ and $\alpha(L,q)+n$ at $q$. Via this identification, we may define $\cgrdab(C,p,q)$ for \emph{all} choices of rank-$r$ ramification sequences $\alpha, \beta$, nonnegative or otherwise; these various spaces are linked by isomorphisms
\begin{equation}
\label{eq:grdabmn}
\mbox{``}+mp+nq\mbox{''}:\ \cgrdab \xrightarrow{\sim} \cG^{r, m+\alpha, n+\beta}_{g,d+m+n}.
\end{equation}
Note also that the numbers $\rho, \hat{\rho}$ are unaffected by these twists; that is,
$$\rho(g,r,d+m+n,m+\alpha,n+\beta) = \rho(g,r,d,\alpha, \beta) \mbox{ and } \hat{\rho}(g,r,d+m+n,m+\alpha,n+\beta) = \hat{\rho}(g,r,d,\alpha, \beta).$$

Therefore we may define ``dimensionally proper points'' in exactly the same way, and $L$ is dimensionally proper in $\cgrdab$ if and only if $L+mp+nq$ is dimensionally proper in $\cG^{r,m+\alpha,n+\beta}_{g,d+m+n}$. The following theorem is immediate.

\begin{thm}
\label{thm:ratlSeries}
Theorem \ref{thm:twoPointBN} is true exactly as written, but without assuming that the ramification sequences are nonnegative, for rational linear series.
\end{thm}

%%%%%%%%%%
%%%%%%%%%%
\subsection{The notation $\cG^{\beta / \alpha}_g$}
As a reminder that we are working with \emph{rational} linear series, and to simplify certain statements, we introduce the following alternate notation.

\begin{defn}
Let $\beta / \alpha$ be a fixed-height skew shape, where $\alpha, \beta$ have rank $r$. Define
$$\cG^{\beta / \alpha}_g = \cG^{r, -\alpha, \beta}_{g, r+g}.$$
\end{defn}

The implicit assumption $\alpha \leq \beta$ is equivalent to saying $\hat{\rho} = \rho$ in Theorem \ref{thm:twoPointBN}. The choice of degree $d = r+g$ is the ``expected'' degree of a line bundle whose complete linear series has rank $r$. While this choice appears restrictive, Equation \eqref{eq:grdabmn} shows us that it is not. This choice of degree has the pleasant consequence that both $r$ and $d$ disappear from the formula for $\rho$, leaving an expression in $g,\alpha,\beta$ alone:
$
\rho(g,r,r+g, -\alpha, \beta) = \hat{\rho}(g,r,r+g,-\alpha,\beta) = g - |\beta / \alpha|.
$

\begin{cor}
\label{cor:gSkew}
The moduli space $\cG^{\beta / \alpha}_g \to \cM_{g,2}$ has expected relative dimension $g - |\beta / \alpha|$. If $g \geq |\beta / \alpha|$, it has dimensionally proper points.
\end{cor}

%%%%%%%%%%
%%%%%%%%%%
\section{Threshold genera of fixed-height skew shapes}
\label{sec:tg}
The existence of dimensionally proper points of $\cG^{\beta / \alpha}_g$ provides a sort of distance, or cost, associated to fixed-height skew shapes $\beta / \alpha$. We continue to allow $\operatorname{char} \FF = p$ in this section.

\begin{defn}
Let $\beta / \alpha$ be a fixed-height skew shape.
The \emph{threshold genus} $\tg(\beta/\alpha)$ is the minimum integer $g_0$ such that for all genera $g \geq g_0$, $\widetilde{\cG}^{\beta / \alpha}_g$ has dimensionally proper points.
\end{defn}

\begin{prop}
\label{prop:tgProps}
Threshold genera of fixed-height skew shapes have the following properties.
\begin{enumerate}
\item (Upper bound) $\tg(\beta / \alpha) \leq |\beta / \alpha|.$
\item (Translation and reflection) $\tg\parenSkew{n+\beta}{n+\alpha} = \tg\parenSkew{\beta}{\alpha}$ and $\tg\parenSkew{n-\alpha}{n-\beta} = \tg\parenSkew{\beta}{\alpha}$.
\item  (Identity) $\tg(\beta / \alpha) = 0$ if and only if $\alpha = \beta$.
\end{enumerate}
\end{prop}

\begin{proof}
Part (1) follows readily from Corollary \ref{cor:gSkew}. Part (2) follows from exchanging marked points and applying twists of the form $L \mapsto L + np - nq$. Part (3) follows by observing that $\PP^1$ is ``a general genus $0$ curve,'' so $G^{\beta / \alpha}(\PP^1, p, q)$ is nonempty if and only if $0 - |\beta / \alpha| = 0$, i.e. $\alpha = \beta$.
\end{proof}

Rectangular skew shapes recover the same threshold genera defined in the introduction.

\begin{lemma}
Assume $\operatorname{char} \FF = 0$.
For any nonnegative integers $g,r,d$ such that $g-d+r \geq 0$,
$\cG^r_{g,d}$ has dimensionally proper points if and only if $\widetilde{\cG}^{(g-d+r)^{r+1} / 0^{r+1}}_g$ has dimensionally proper points.
\end{lemma}

\begin{proof}
By Equation \eqref{eq:grdabmn}, the latter is isomorphic to $\widetilde{\cG}^{r, 0^{r+1}, 0^{r+1}}_{g, d}$. In characteristic $0$, any linear series has only finitely many ramification points, so $\widetilde{\cG}^{r, 0^{r+1}, 0^{r+1}}_{g, d}$ is dense in ${\cG}^{r, 0^{r+1}, 0^{r+1}}_{g, d}$.
Now, ${\cG}^{r, 0^{r+1}, 0^{r+1}}_{g, d}$ simply parameterizes points of $\cG^r_{g,d}$ together with two unconstrained marked points on the curve, and it follows that it has dimensionally proper points if and only if $\cG^r_{g,d}$ does.
\end{proof}

\begin{cor}
\label{cor:skewRect}
Assume $\operatorname{char} \FF = 0$. For positive integers $a,b$, $\tg(a \times b) = \tg(b^a / 0^a) = \tg(a^b / 0^b)$.
\end{cor}

%%%%%%%%%%
%%%%%%%%%%
\section{The subadditivity theorem}
\label{sec:subadd}

The engine driving this whole paper is the following sub-additivity theorem, which justifies our method of bounding threshold genus by ``lego-building,'' and shows that threshold genus of skew shapes behaves like a distance function. In this theorem and section, we allow $\operatorname{char} \FF = p$.

\begin{thm}
\label{thm:subadd}
For any fixed-height ramification sequences $\alpha \leq \beta \leq \gamma$, 
$$\tg(\gamma/\alpha) \leq \tg(\gamma/\beta) + \tg(\beta/\alpha).$$
\end{thm}

The idea behind the proof of Theorem \ref{thm:subadd} is illustrated in Figure \ref{fig:subadd}: from dimensionally proper points of $\widetilde{\cG}^{\beta / \alpha}_{g_1}$ and $\widetilde{\cG}^{\gamma / \beta}_{g_2}$, we obtain a nodal curve by gluing along marked points, and construct a \emph{limit linear series} on this nodal curve, which can then be smoothed. The result which will provide the necessary smoothing is the following, from which we will deduce Theorem \ref{thm:subadd}.

\begin{thm}
\label{thm:regen}
Let $g_1, g_2, d,r \geq 0$ be integers, and $\alpha, \beta, \gamma$ three rank-$r$ nonnegative ramification sequences, with $d-r-\beta$ also nonnegative. Assume furthermore that 
$\rho(g_1,r,d,\alpha, \beta) \leq 0 \mbox{ and } \rho(g_2,r,d,d-r-\beta,\gamma) \leq 0.$
If $\widetilde{\cG}^{r, \alpha, \beta}_{g_1,d}$ and $\widetilde{\cG}^{r,d-r-\beta, \gamma}_{g_2,d}$ both have dimensionally proper points, then $\widetilde{\cG}^{r,\alpha,\gamma}_{g_1+g_2, d}$ has dimensionally proper points. 
\end{thm}

\subsection{Regeneration of limit linear series}
\label{ssec:regen}

Theorem \ref{thm:regen} is not novel; it is a straightforward application of the theory of limit linear series, as developed in \cite{eh86}, and especially the Smoothing Theorem 3.4, which is also called the ``Regeneration Theorem'' in the exposition \cite[Theorem 5.41]{hm}. See also the development of Osserman \cite{oss06} which requires no assumptions on characteristic, and the more modern treatment discussed in \cite{lieblichOsserman} and the references therein. A nearly identical statement is \cite[Lemma 3.8]{nonprimitive}, which is proved based on the development in \cite{oss06}, though that paper uses a weaker notion of ``dimensionally proper'' that does not demand that fiber dimension be minimal. We briefly sketch the ideas here, but the reader should consult these references for details and background.

Let $X$ be a curve of compact type, i.e. a nodal curve in which every node is disconnecting. Let $C_1, \cdots, C_n$ be the components of $X$. A \emph{limit linear series} of degree $d$ and rank $r$ on $X$ is an $n$-tuple $L = (L_i)_{1 \leq i \leq n}$, where $[L_i] \in G^r_d(C_i)$, subject to the compatibility condition that if $p \in X$ is a node lying on components $C_i$ and $C_j$, then $\alpha(L_i, p) \geq d-r - \alpha(L_j,p)$. Call $L$ \emph{refined} if equality holds at every node. For a smooth point $q \in X$, define the ramification sequence of $L$ at $q$ to be $\alpha(L,q) = \alpha(L_i, q)$, where $C_i$ is the component on which $q$ lies. The set of limit linear series forms a subvariety $G^r_d(X) \subseteq \prod_{i=1}^n G^r_d(C_i)$, in which the refined series form an open locus, and one may define a subvariety $G^{r,\alpha,\beta}_d(X,p,q)$ imposing ramification at two smooth points. The same bounds $\rho(g,r,d)$ and $\rho(g,r,d,\alpha,\beta)$ are valid lower bounds on dimension for limit linear series as well.

The construction of limit linear series may be relativized to deformations of a curve of compact type, and the Brill--Noether number $\rho$ (with or without imposed ramification) remains a lower bound on \emph{relative} dimension, but one must be careful: current techniques only construct these relative spaces of limit linear series for certain deformations called \emph{smoothing families.} One may need to pass to an \'etale neighborhood to obtain a smoothing family. This difficulty meant that for many years, applications of limit linear series required some boilerplate language about passing from the family you \emph{really} care about to a smoothing family. This did not prevent the theory from having useful applications, but it had the frustrating consequence that there was no global moduli space of limit linear series, and also stymied applications to non-algebraically closed  fields. These difficulties were surmounted by Lieblich and Osserman in \cite{lieblichOsserman} using their formalism of descent of moduli spaces; we now have a global moduli space $\overline{\cG}^r_{g,d} \to \overline{\cM}_g^{\operatorname{ct}}$ over the moduli of curves of compact type, with the catch that this morphism is not known to be representable by schemes, but only by algebraic spaces. This global moduli space obeys the bound $\rho$ on local relative dimension. Imposing ramification conditions, we obtain a moduli space $\overline{\cG}^{r,\alpha,\beta}_{g,d} \to \overline{\cM}_{g,2}^{\operatorname{ct}}$ obeying the local relative dimension bound $\rho(g,r,d,\alpha,\beta)$. These tools quickly prove Theorem \ref{thm:regen}.

\begin{proof}[Proof of Theorem \ref{thm:regen}]
Let $L_1, L_2$ be linear series on twice-marked curves $(C_1, p_1, q_1), (C_2, p_2, q_2)$ that are dimensionally proper in $\widetilde{\cG}^{r,\alpha, \beta}_{g_1}$ and $\widetilde{\cG}^{r, d-r-\beta, \gamma}_{g_2}$, respectively. Let $g = g_1+g_2$, and let $(X,p_1, q_2)$ be the twice-marked nodal curve of genus $g$ obtained by gluing $q_1$ to $p_2$; then $L = (L_1, L_2)$ is a refined limit linear series on $(X,p_1, q_2)$, with ramification exactly $\alpha, \gamma$ at $p_1, q_2$. It is isolated in $G^{r,\alpha,\beta}_{d}(X,p_1, q_2)$, since $L_1, L_2$ are isolated in $G^{r, \alpha, \beta}_d(C_1, p_1, q_1)$ and $G^{r, d-r-\beta, \gamma}_d(C_2, p_2,q_2)$. Now, denote by $\partial \overline{\cG}^{r,\alpha, \gamma}_{g,d}$ the part of $\overline{\cG}^{r,\alpha, \gamma}_{g,d}$ lying over the boundary of $\overline{\cM}^{\operatorname{ct}}_{g,2}$. Then we have a map $\widetilde{\cG}^{r,\alpha, \beta}_{g_1,d} \times \widetilde{\cG}^{r, d-r-\beta, \gamma}_{g_2, d} \hookrightarrow  \partial \overline{\cG}^{r,\alpha, \gamma}_{g,d}$, whose image provides a neighborhood of $L$. Since $L_1, L_2$ are dimensionally proper, the local relative dimension (relative to the boundary of $\overline{\cM}_{g,2}^{\operatorname{ct}}$) at $L$ is $\rho(g_1,r,d,\alpha,\beta) + \rho(g_2,r,d,d-r-\beta,\gamma) = \rho(g,r,d,\alpha,\gamma)$. Since the local relative dimension of $L$ in $\partial \overline{\cG}^{r,\alpha, \gamma}_{g,d}$ cannot be smaller than the local relative dimension in $\overline{\cG}^{r,\alpha, \gamma}_{g,d}$, it follows that the same is true in $\overline{\cG}^{r,\alpha, \gamma}_{g,d}$. But this means that $\overline{\cG}^{r,\alpha, \gamma}_{g,d}$ cannot be supported over the boundary; in a neighborhood of $L$ there must be linear series over \emph{smooth} curves. By semicontinuity, all such series in a small enough neighborhood have ramification exactly $\alpha, \gamma$ at the marked points, are isolated in their fibers, and have local relative dimension $\rho$, i.e. they are dimensionally proper points of $\cG^{r, \alpha, \gamma}_{g,d}$.
\end{proof}

\subsection{From regeneration to subadditivity}
We now deduce the subadditivity theorem.

\begin{proof}[Proof of Theorem \ref{thm:subadd}]
We must show that $\widetilde{\cG}^{\gamma / \alpha}_g$ has dimensionally proper points for all $g$ such that $\tg(\gamma / \beta) + \tg(\beta / \alpha) \leq g \leq |\gamma / \alpha|$. Fix such a $g$.
Choose two integers $g_1, g_2$ such that $\tg(\beta / \alpha) \leq g_1 \leq |\beta / \alpha|$ and $\tg(\gamma / \beta) \leq g_2 \leq |\gamma / \beta|$. By definition of threshold genus, both $\widetilde{\cG}^{\beta / \alpha}_{g_1}$ and $\widetilde{\cG}^{\gamma / \beta}_{g_2}$ have dimensionally proper points. We now use the isomorphism \eqref{eq:grdabmn}: we will choose $N \gg 0$ and twist a dimensionally proper point of $\widetilde{\cG}^{\beta / \alpha}_{g_1}$ by $Np_1 + (N+g_2)q_1$ and a dimensionally proper point of  $\widetilde{\cG}^{\gamma / \beta}_{g_2}$ by $(N+g_1)p_2 + N q_2$. For $N$ large, we obtain \emph{regular} linear series, and the curious choice of twists ensures that we obtain series of degree $d' = r+2N + g_1 + g_2$ in both cases.  It follows that both $\widetilde{\cG}^{r, N-\alpha, N+g_2 + \beta}_{g_1, d'}$ and $\widetilde{\cG}^{r,N+g_1-\beta, N + \gamma}_{g_2, d'}$ have dimensionally proper points. Since $d'-r-(N+g_2 + \beta) = N + g_1 -\beta$, and each of $N-\alpha, N+g_2+\beta, N+g_1 - \beta, N+\gamma$ is a nonnegative ramification sequence, it follows from Theorem \ref{thm:regen} that $\widetilde{\cG}^{r, N - \alpha, N+\gamma}_{g, d'}$ has dimensionally proper points. Using isomorphism \eqref{eq:grdabmn} again to twist down by $-Np - Nq$, it follows that $\widetilde{\cG}^{\gamma / \alpha}_{g}$ has dimensionally proper points. Therefore $\tg(\gamma / \alpha) \leq  \tg(\gamma / \beta) + \tg(\beta / \alpha)$.
\end{proof}

\subsection{Subadditivity for rectangles}
\label{ssec:subaddRect}

For positive integers $a,b,c$, invariance under translations and the subadditivity theorem imply that 
$$\tg( (b+c)^a / 0^a) \leq \tg((b+c)^a / b^a) + \tg(b^a / 0^a) = \tg(c^a / 0^a) + \tg(b^a / 0^a).$$
Therefore Corollary \ref{cor:skewRect} and the symmetry of $\tg(a \times b)$ in $a,b$ imply

\begin{cor}
\label{cor:rectSubadd}
Assume $\operatorname{char} \FF = 0$. For positive integers $a,b,c$, $\tg(a \times (b+c)) \leq \tg(a \times b) + \tg(a \times c)$ and $\tg( (a+b) \times c) \leq \tg(a \times c) + \tg(b \times c)$.
\end{cor} 

%%%%%%%%%%
%%%%%%%%%%
\section{Displacement difficulty of fixed-height skew shapes}
\label{sec:difficulty}

A crucial source of ``cheap lego bricks'' in the proof of Theorem \ref{thm:main} are fixed-height skew shapes $\beta / \alpha$ with threshold genus $1$. We will characterize these via a curious combinatorial construction called \emph{displacement}. Such skew shapes can be assembled like lego bricks, thereby obtaining an upper bound on threshold genera that is completely combinatorial, and computable in principle. In this section, we continue to allow $\operatorname{char} \FF = p$.

The key idea ideas in this section were inspired by \cite{eh87}; this work began by observing that the method in that paper could be applied not just to the ramification of the canonical series.

%%%%%%%%%%
%%%%%%%%%%
\subsection{Displacement of ramification sequences}
\label{ssec:displacement}

We begin with some terminology.

\begin{defn}
An arithmetic progression, for purposes of this paper, is a proper subset $\Lambda \subsetneq \ZZ$, that is either empty or of the form $n + m \ZZ$ for $m,n \in \ZZ$ with $m =0$ or $m \geq 2$. 
In the latter case, $m$ is called the \emph{modulus} of $\Lambda$.
In particular, we allow the empty set, we allow $m=0$ and view single-element sets as arithmetic progressions, but we do \emph{not} allow $m=1$.
\end{defn}

%These ``arithmetic progressions'' are precisely the sets that arise as $\Lambda = \{n \in \ZZ: np \sim nq +D \}$, where $p,q$ are two distinct points on a genus-$1$ curve. 

\begin{defn}
Let $S$ be a set of integers, and $\Lambda$ an arithmetic progression. The \emph{upward displacement of $S$ along $\Lambda$}, denoted $\dul(S)$, is the set obtained from $S$ by replacing $n\in S$ by $n+1$ whenever $n+1 \in \Lambda $ and $n +1 \not\in S$. The \emph{downward displacement of $S$ along $\Lambda$}, denoted $\ddl(S)$ is the set obtained from $S$ by replacing $n\in S$ by $n-1$ whenever $n \in \Lambda$ and $n-1 \not\in S$.
\end{defn}

A useful way to regard to operations $\dul, \ddl$ is: whenever $n \in \Lambda$ and $\{n-1,n\} \cap S$ has exactly one element, that element is ``loose;'' $\dul$ slides loose elements up (attracting to $\Lambda$), while $\ddl$ slides loose elements down (repelling away from $\Lambda$). The following properties are immediate.

\begin{lemma}
\label{lem:idempotent}
For any arithmetic progression $\Lambda$, the operations $\dul, \ddl$ satisfy the identities
$$\dul \circ \dul = \dul \circ \ddl = \dul \mbox{ and } \ddl \circ \ddl = \ddl \circ \dul = \ddl.$$
\end{lemma}

A ramification sequence is uniquely determined by the set of elements in its associated vanishing sequence. Using this correspondence, we extend $\dul,\ddl$ from sets to ramification sequences.

\begin{defn}
Let $\alpha$ be a ramification sequence. Call an entry $\alpha_i$ \emph{increasable} (resp. \emph{decreasable}) if $\alpha$ is still a nondecreasing sequence after $\alpha_i$ is increased (resp. decreased) by $1$.

For an arithmetic progression $\Lambda$, define ramification sequences $\dul(\alpha), \ddl(\alpha)$ as follows.
\begin{eqnarray*}
\dul(\alpha)_i &=& \begin{cases}
\alpha_i +1 & \mbox{ if } \alpha_i \mbox{ is increasable and } \alpha_i+i+1 \in \Lambda\\
\alpha_i & \mbox{ otherwise}
\end{cases}\\
\ddl(\alpha)_i &=& \begin{cases}
\alpha_i -1 & \mbox{ if } \alpha_i \mbox{ is decreasable and } \alpha_i+i \in \Lambda\\
\alpha_i & \mbox{ otherwise}
\end{cases}
\end{eqnarray*}
Equivalently, if $a$ is the associated vanishing sequence of $\alpha$, regarded as a set, then the associated vanishing sequences of $\dul(\alpha)$ and $\ddl(\alpha)$ are $\dul(a)$ and $\ddl(a)$.
\end{defn}

See Figure \ref{fig:disp} for a useful visualization: if all elements of $\alpha$ are positive and we draw $\alpha$ as a Young diagram, the arithmetic progression $\Lambda$ may be drawn as a sequence of equally space diagonal lines, and $\dul$ has the effect of adding ``addable'' boxes (within the fixed height of the diagram) on any of these lines, while $\ddl$ removes any ``removable'' boxes.

\begin{eg}
For any of the four ramification sequences $\alpha \in \{ (2,2,4), (2,3,4), (2,2,5), (2,3,5) \}$, we have $\dul(\alpha) = (2,3,5)$ and $\ddl(\alpha) = (2,2,4)$. 
\end{eg}

\begin{figure}

\begin{tikzpicture}
% Modified the tikzcode auto-generated for the figure below.
\draw (0.000,0.000) rectangle ++(0.400,0.400);
\draw (0.400,0.000) rectangle ++(0.400,0.400);
\draw (0.000,0.400) rectangle ++(0.400,0.400);
\draw (0.400,0.400) rectangle ++(0.400,0.400);
\draw (0.000,0.800) rectangle ++(0.400,0.400);
\draw (0.400,0.800) rectangle ++(0.400,0.400);
\draw (0.800,0.800) rectangle ++(0.400,0.400);
\draw (1.200,0.800) rectangle ++(0.400,0.400);
\draw (0.6,-1) node {$\alpha = (2,2,4)$};

\begin{scope}[shift={(4,0)}]
% Figure generated by calling dispDiagram([2, 2, 4],[1, 4, 7]) from figures.py
\draw (0.000,0.000) rectangle ++(0.400,0.400);
\draw (0.400,0.000) rectangle ++(0.400,0.400);
\draw (0.000,0.400) rectangle ++(0.400,0.400);
\draw (0.400,0.400) rectangle ++(0.400,0.400);
\draw (0.000,0.800) rectangle ++(0.400,0.400);
\draw (0.400,0.800) rectangle ++(0.400,0.400);
\draw (0.800,0.800) rectangle ++(0.400,0.400);
\draw (1.200,0.800) rectangle ++(0.400,0.400);
\draw (-0.400,0.800) -- (0.800,-0.400);
\draw (0.000,1.600) -- (2.000,-0.400);
\draw[pattern=north east lines] (0.800,0.400) rectangle (1.200,0.800);
\draw (1.200,1.600) -- (2.400,0.400);
\draw[pattern=north east lines] (1.600,0.800) rectangle (2.000,1.200);
\draw (0.6,-1) node {$\Lambda = 1 + 3 \ZZ$};
\end{scope}

\begin{scope}[shift={(8,0)}]
% Modified auto-generated commands above
\draw (0.000,0.000) rectangle ++(0.400,0.400);
\draw (0.400,0.000) rectangle ++(0.400,0.400);
\draw (0.000,0.400) rectangle ++(0.400,0.400);
\draw (0.400,0.400) rectangle ++(0.400,0.400);
\draw (0.000,0.800) rectangle ++(0.400,0.400);
\draw (0.400,0.800) rectangle ++(0.400,0.400);
\draw (0.800,0.800) rectangle ++(0.400,0.400);
\draw (1.200,0.800) rectangle ++(0.400,0.400);
\draw (0.800,0.400) rectangle (1.200,0.800);
\draw(1.600,0.800) rectangle (2.000,1.200);
\draw (0.6,-1) node {$\dul(\alpha) = (2,3,5)$};
\end{scope}
\end{tikzpicture}

\caption{
An example of the upward displacement operation on ramification sequences.
}
\label{fig:disp}
\end{figure}
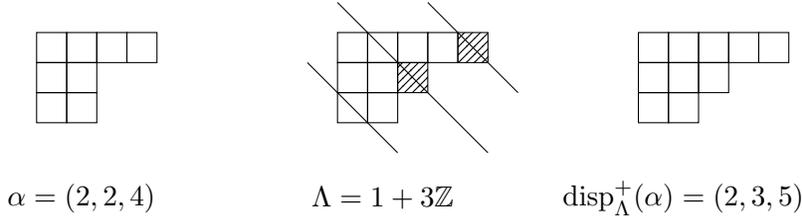

\begin{rem}
\label{rem:fixedHeight}
This construction is closely related to, but subtly different from, the displacement operations on \emph{partitions} defined in \cite{pflChains} and the unpublished preprint \cite{negativeRho13}. A partition can be represented as a nonincreasing sequence $\lambda = (\lambda_0, \lambda_1, \cdots)$ of integers, almost all $0$, and $\dul(\lambda), \ddl(\lambda)$ can be defined via displacement of the infinite set $\{\lambda_n - n-1: n \geq 0\}$. A nonnegative ramification sequence $\alpha = (\alpha_0, \cdots, \alpha_r)$ of course determines a partition $\lambda = (\alpha_r, \alpha_{r-1}, \cdots, \alpha_0, 0, 0, \cdots)$, and displacements of this partition \emph{almost} correspond to displacements of the ramification sequence, with two crucial differences: the first $0$ entry in a partition must be considered increasable, and the last entry of a ramification sequence must be considered decreasable, even when it is $0$.
This causes subtle but important differences when analyzing the two situations.
This distinction is the reason this paper emphasizes the phrase ``fixed-height'' when discussing skew shapes.
\end{rem}

\begin{defn}
Two ramification sequences $\alpha \leq \beta$ are \emph{linked by $\Lambda$} if $\alpha =  \ddl(\beta)$ and $\beta = \dul(\alpha)$. 
We say that $\alpha, \beta$ are \emph{linked} if they are linked by some arithmetic progression, and $\beta / \alpha$ is called an $n$-link if $\alpha, \beta$ are linked and $|\beta / \alpha| = n$. 
\end{defn}

\begin{eg}
\label{eg:2link}
The sequences $\alpha = (2,2,4), \beta = (2,3,5)$ are a $2$-link, since for $\Lambda = 1 + 3 \ZZ$ we have $\dul((2,2,4)) = (2,3,5)$ and $\ddl((2,3,5)) = (2,2,4)$; see Figure \ref{fig:disp}. However, $(0,0,2,3,5) / (0,0,2,2,4)$ is \emph{not} a $2$-link, because the first element is decreasable; if $\Lambda$ is a progression with $\dul(0,0,2,2,4) \geq (0,0,2,3,5)$, then necessarily $\Lambda = 3 \ZZ$ and $\ddl((0,0,2,2,4)) = (-1,0,2,2,4)$. 

Therefore padding ramification sequences with $0$s need not preserve links.
This illustrates the issue in Remark \ref{rem:fixedHeight}.
\end{eg}

\begin{lemma}
\label{lem:1links}
Let $\alpha \leq \beta$. Then $\beta / \alpha$ is a $1$-link if and only if $|\beta / \alpha| = 1$.
\end{lemma}
\begin{proof}
If $|\beta / \alpha| = 1$, then for the single index $j$ where $\alpha_j < \beta_j$, the single-element progression $\Lambda = \{ \alpha_i+i+1\}$ gives $\beta = \dul(\alpha)$ and $\alpha = \ddl(\beta)$. 
\end{proof}

\begin{defn}
The \emph{loose set} of a ramification sequence  $\alpha$ is the following set of integers.
$$
\{ \alpha_i + i:\ \alpha_i \mbox{ decreasable} \} 
\cup
\{ \alpha_i +i + 1:\ \alpha_i \mbox{ increasable} \}
$$
\end{defn}

\begin{lemma}
\label{lem:2links}
For two ramification sequences $\alpha \leq \beta$ of rank $r$, $\beta / \alpha$ is a $2$-link if and only if 
\begin{enumerate}
\item There exist two distinct indices $0 \leq j < k \leq r$ such that $\beta_j = \alpha_j +1, \beta_k = \alpha_k+1$, and $\beta_i = \alpha_i$ for all other indices $i$; and
\item Denoting by $\Lambda$ the minimal arithmetic progression containing $\{\alpha_j+j+1, \alpha_k+k+1\}$, i.e.
$$\Lambda = \{ n:\ n \equiv \alpha_j + j +1  \pmod{ \alpha_j + j - \alpha_k - k}\},$$
the only elements of $\Lambda$ in the loose set of $\alpha$ are $\alpha_j + j+1$ and $\alpha_k + k+1$.
\end{enumerate}
\end{lemma}

\begin{proof}
If these two conditions hold, then $|\beta / \alpha| = 2$, and it follows from definitions that $\ddl(\beta) = \alpha$ and $\dul(\alpha) = \beta$, so $\beta / \alpha$ is a $2$-link. Conversely, if $\beta / \alpha$ is a $2$-link then condition (1) follows from the assumption $| \beta / \alpha | = 2$ and the fact that upward displacement cannot increase any elements by more than $1$. So there exists \emph{some} arithmetic progression $\Lambda'$ linking $\alpha$ to $\beta$. This $\Lambda'$ must contain $\Lambda$, and $\Lambda'$ can only meet the loose set in the two specified values, so that same is true of $\Lambda$.
\end{proof}

\begin{eg}
The loose set of $\alpha = (2,2,4)$ is $\{2,4,6,7\}$. The fact that $(2,3,5) / (2,2,4)$ is a $2$-link (Example \ref{eg:2link}) follows form the fact that $\alpha_1+1+1 = 4$ and $\alpha_2+2+1 = 7$ generates the arithmetic progression $1 + 3 \ZZ$, which includes no other elements of the loose set.
\end{eg}

\begin{lemma}
\label{lem:linkedCrit}
Let $\alpha < \beta$ be ramification sequences, and $\Lambda$ an arithmetic progression.
The following are equivalent.
\begin{enumerate}
\item $\alpha, \beta$ are linked by $\Lambda$.
\item There exists $\gamma$ with $\alpha \leq \gamma \leq \beta$, $\alpha = \ddl(\gamma)$ and $\beta = \dul(\gamma)$.
\item Every $\gamma$ with $\alpha \leq \gamma \leq \beta$ satisfies $\alpha = \ddl(\gamma)$ and $\beta = \dul(\gamma)$.
\end{enumerate}
\end{lemma}

\begin{proof}[Proof]
Suppose (1) holds. Since $\dul$ is idempotent (Lemma \ref{lem:idempotent}) it follows that $\dul(\beta) = \dul(\dul(\alpha)) = \dul(\alpha) = \beta$. So we may take $\gamma = \beta$ and (2) holds. So (1) implies (2).

Suppose (2) holds, with $\alpha = \ddl(\gamma), \beta = \dul(\gamma)$. Applying $\dul$ to both equations, Lemma \ref{lem:idempotent} implies that $\beta = \dul(\beta) = \dul(\alpha)$; similarly, applying $\ddl$ to both equations implies $\alpha = \ddl(\alpha) = \ddl(\beta)$. Suppose $\gamma'$ also satisfies $\alpha \leq \gamma' \leq \beta$. Observe the displacement is monotonic in $\leq$; this implies that $\dul(\alpha) \leq \dul(\gamma') \leq \dul(\beta)$ so $\dul(\gamma') = \beta$. Similarly, $\ddl(\alpha) \leq \ddl(\gamma') \leq \ddl(\beta)$ and therefore $\ddl(\gamma') = \alpha$. So (2) implies (3).

Finally, (3) implies (1) by choosing $\gamma = \alpha$ and $\gamma = \beta$.
\end{proof}

%%%%%%%%%%
%%%%%%%%%%
\subsection{A characterization of threshold genus $1$}
\label{ssec:tg1}

\begin{thm}
\label{thm:gdOne}
For any fixed-height skew shape $\beta / \alpha$, $\tg(\beta/\alpha) = 1$ if and only if $\beta / \alpha$ is either a $1$-link or a $2$-link.
\end{thm}

We prove this theorem over the course of this subsection. The core of the argument is the following lemma. This lemma is essentially a rephrasing of \cite[Proposition 5.2]{eh87}, but we include a proof because there is a subtle error in the statement of that Proposition that was immaterial in that paper but would cause confusion in our present context
\footnote{The error in \cite[Proposition 5.2]{eh87}, in that paper's notation, is that hypothesis (**) should also include the implication ``$(a_{r-i}+1)p + b_i q \sim D \Rightarrow b_{i+1} = b_i + 1$.'' As stated, this proposition has the following counterexample: suppose $p-q$ is $2$-torsion, $d=3$, $a = (0,3)$, $b=(0,2)$, and $D = 3p \sim p + 2q$. Then hypothesis (**) holds, yet no linear series $(\cO_E(D), V)$ has vanishing orders exactly $a,b$ at $p,q$, since such a series would necessarily be spanned by the two divisors $3p$ and $p+2q$ and therefore have a base point at $p$.
}.

\begin{lemma}
\label{lem:Lgamma}
Fix a smooth twice-marked genus $1$ curve $(E,p,q)$, and a ramification sequence $\gamma$ of rank $r$. 
For every line bundle $\cL$ on $E$ of degree $r+1$, there exists a unique subspace $V(\gamma) \subseteq H^0(E \backslash \{p,q\}, \cL)$ such that $(\cL, V(\gamma)) \in G^{\gamma / \gamma}(E,p,q)$. Let $L(\gamma) = (\cL,V(\gamma))$. The actual ramification sequences of $L(\gamma)$ at $p$ and $q$ are $-\ddl(\gamma)$ and $\dul(\gamma)$, respectively, where $\Lambda$ is the arithmetic progression $\{ n \in \ZZ: \cL \cong \cO_E((r+1-n)p + nq) \}$.
\end{lemma}

\begin{proof}
We will construct $V(\gamma)$ explicitly. Let $c = (c_0, \cdots, c_r)$ be the vanishing orders corresponding to the ramification sequence $\gamma$.   The vanishing sequence corresponding to ramification sequence $-\gamma$ is $(r-c_r, r-c_{r-1}, \cdots, r-c_0)$.
The basic observation is that for any rational linear series $(\cL, V)$ with ramification at least $-\gamma$ at $p$ and $\gamma$ at $q$, and any $0 \leq i \leq j \leq r$, the subspace of $V$ consisting of sections vanishing to order at least $r-c_{j}$ at $p$ at order at least $c_i$ at $q$ has dimension at least $(r+1) - (r-j) - i = j-i+1$.
Furthermore, this subspace must also be a subspace of the image of $H^0(E, \cL(-(r-c_j)p-c_iq)) \hookrightarrow H^0(E \backslash\{p,q\}, \cL)$, which has dimension $c_j-c_i + 1$ by Riemann-Roch. So whenever $j-i = c_j-c_i$, or equivalently $\gamma_i = \gamma_j$, this subseries is uniquely determined.
 
Partition $\{0, \cdots, r\}$ into disjoint subsets such that $i,j$ are grouped together if $\gamma_i = \gamma_j$. Call these subsets \emph{blocks}. If $\{i, i+1, \cdots, j\}$ is a block, then $c_i, \cdots, c_j$ are consecutive, so $c_j - c_i = j-i$, and neither $c_i-1$ nor $c_j+1$ occur in the vanishing sequence $c$. For such a block, define $V_{i:j}$ to be the image of the natural inclusion $H^0(E, \cL(-(r-c_j)p-c_iq)) \hookrightarrow H^0(E \backslash \{p,q\}, \cL)$. As observed in the previous paragraph, $\dim V_{i:j}$ is equal to the size of the block, and $V_{i:j} \subseteq V$ for any $(\cL, V)$ with the desired ramification. 

It follows from Riemann-Roch that the vanishing orders of $V_{i:j}$ at $q$ are $c_i, \cdots, c_{j-1}, c_j'$, where $c_j' = c_j+1$ if $c_j + 1 \in \Lambda$ and $c_j' = c_j$ otherwise.

Since $c_i, \cdots, c_j$ are consecutive and $c_j+1$ is not in $c$, these numbers are precisely $\dul(c)_i, \cdots, \dul(c)_j$. Similarly, the fact that $c_i-1$ is not in the vanishing sequence implies that the vanishing orders of $V_{i:j}$ at $p$ are precisely $r-\ddl(c)_j, \cdots r - \ddl(c)_i$.

Define $V(\gamma)$ to be the sum, taken over all blocks $\{i, i+1, \cdots, j\} \subseteq \{0, \cdots r\}$, of $V_{i:j}$. The vanishing orders at $p$ of sections in these various subspaces are all disjoint, so this sum is a direct sum, $\dim V(\gamma) = r+1$, and the vanishing sequences of $V(\gamma)$ at $p$ and $q$ are precisely $r-\ddl(c)$ and $\dul(c)$, respectively. In other words, its ramification sequences at $p,q$ are $-\ddl(\gamma), \dul(\gamma)$, and $(\cL, V(\gamma))$ is a rational linear series of the desired form. Since each summand $V_{i:j}$ must be a subspace of $V$ for any rational linear series $(\cL,V)$ with the desired ramification, $V(\gamma)$ is the only possible such subspace.
\end{proof}

\begin{prop}
\label{prop:BsingleElliptic}
Let $(E,p,q)$ be a smooth twice-marked genus $1$ curve. Let $m$ be the order of $p-q$ in the Jacobian, where $m=0$ if $p-q$ is nontorsion. For any fixed-height skew shape $\beta / \alpha$,
\begin{enumerate}
\item If $\alpha = \beta$, then $\widetilde{G}^{\beta / \alpha}(E,p,q)$ is $1$-dimensional.
\item If $\alpha < \beta$, then $\widetilde{G}^{\beta / \alpha}(E,p,q)$ has a single point if $\alpha, \beta$ are linked by an arithmetic progression of modulus $m$, and it is empty otherwise.
\end{enumerate}
\end{prop}

\begin{proof}
Consider the fibers of the forgetful map $\widetilde{G}^{\beta / \alpha}(E,p,q) \to \Pic^{r+1}(E)$, where $r$ is the rank of $\alpha$ and $\beta$. For any $[\cL] \in \Pic^{r+1}(E)$, any point $L = (\cL, V)$ in the fiber is a rational linear series with ramification exactly $-\alpha$ at $p$ and $\beta$ at $q$; since $\beta \geq \alpha$ it follows from Lemma \ref{lem:Lgamma} that $L$ must be $L(\alpha)$.
So every fiber is either empty or a single point, and it is nonempty if and only if $\ddl(\alpha) = \alpha$ and $\dul(\alpha) = \beta$, where $\Lambda = \{n \in \ZZ: \cL \cong \cO_E((r+1-n)p + n q \}$. 
This last condition is equivalent to $\alpha, \beta$ being linked by $\Lambda$, by Lemma \ref{lem:linkedCrit}.
As $\cL$ varies, every possible arithmetic progression $\Lambda$ with modulus $m$ occurs exactly once, and for all other choices of $\cL$, the progression $\Lambda$ is empty. If $\alpha = \beta$, this means that the fiber is nonempty for all choices of $\cL$ except a finite set corresponding to the decreasable and increasable elements of $\alpha$; part (1) follows. Now assume $\alpha < \beta$. If $\alpha, \beta$ are not linked by any arithmetic progressions with modulus $m$, then all fibers are empty. On the other hand, if they are linked by such a progression, then they are linked by a unique such progression, namely $\alpha_i + i + 1 + m \ZZ$, where $i$ is any index for which $\alpha_i < \beta_i$, so there is a single nonempty fiber and $\widetilde{G}^{\beta / \alpha}(E,p,q)$ is a single point.
\end{proof}

\begin{prop}
\label{prop:dpLinks}
For every fixed-height skew shape $\beta / \alpha$, the moduli space $\widetilde{\cG}^{\beta / \alpha}_{1}$ has dimensionally proper points if and only if $\beta / \alpha$ is an $n$-link for some $n \in \{0,1,2\}$.
\end{prop}

\begin{proof}
Proposition \ref{prop:BsingleElliptic} shows that $\widetilde{\cG}^{\beta / \alpha}_1$ is empty when $\beta / \alpha$ is not a link, so we may assume that $\beta / \alpha$ is an $n$-link for some $n$. 
If $n \in \{0,1\}$, the result follows from Corollary \ref{cor:gSkew}.

Consider the case $n=2$. Choose an integer $m$ such that $\alpha$ is linked to $\beta$ by an arithmetic progression with modulus $m$, and let $(E,p,q)$ have torsion order $m$. Consider a family $(\cE,p,q)$ of twice-marked genus $1$ over a base scheme $S$ that includes $(E,p,q)$, but for which a general member has torsion order $0$. For example, one may take $S = E \backslash \{p\}$ and consider a family in which $p$ is fixed and $q$ moves to all other points on $E$. There are a finite set of integers $m'$ for which $\alpha, \beta$ are linked by a progression modulo $m'$, and for each one the locus in $S$ where this torsion order occurs has codimension $1$. So Proposition \ref{prop:BsingleElliptic} implies that the image of $\widetilde{G}^{\beta / \alpha}(\cE,p,q) \to S$ has codimension $1$, over which all fibers are $0$-dimensional. It follows that $\widetilde{G}^{\beta / \alpha}(\cE,p,q)$ has a component of dimension $\dim S - 1 = \dim S + 1 - |\beta/\alpha|$, all points of which are isolated in their fibers; hence this component is dimensionally proper. So $\widetilde{\cG}^{\beta / \alpha}_1$ has dimensionally proper points, by the discussion in Section \ref{ssec:thmRemarks}.

Finally, suppose $n \geq 3$. For any family of twice-marked genus-$1$ curves $(\cE,p,q)$ over a base scheme $S$ and any point of $\widetilde{G}^{\beta / \alpha}(\cE,p,q)$ over $x \in S$, there is a codimension-$1$ locus in $S$ where the same torsion order occurs, and therefore the local dimension of $\widetilde{G}^{\beta / \alpha}(\cE,p,q)$ is at least $\dim S - 1$ at this point. Since $\dim S -1 > \dim S + 1 - |\beta / \alpha|$, this point cannot be dimensionally proper. So $\widetilde{\cG}^{\beta/\alpha}_1$ cannot have dimensionally proper points in this case.
\end{proof}

\begin{proof}[Proof of Theorem \ref{thm:gdOne}]
First, suppose that $\tg(\beta / \alpha) = 1$. Then $\widetilde{\cG}^{\beta / \alpha}_1$ has dimensionally proper points, so $\beta / \alpha$ is an $n$-link for $n \in \{0,1,2\}$. If $n=0$ then $\alpha = \beta$, and by Proposition \ref{prop:tgProps} $\tg(\beta / \alpha) = 0$. So in fact $n \in \{1,2\}$. Conversely, suppose that $\beta / \alpha$ is an $n$-link for some $n \in \{1,2\}$. Then $\tg(\beta / \alpha) \geq 1$ by Proposition \ref{prop:tgProps}, so it suffices to verify that $\widetilde{\cG}^{\beta / \alpha}_g$ has dimensionally proper points for all $g \geq 1$. Corollary \ref{cor:gSkew} shows that $\widetilde{\cG}^{\beta / \alpha}_g$ has dimensionally proper points for all $g \geq n$, and Proposition \ref{prop:dpLinks} shows that such points exist when $g=1$.
\end{proof}

%%%%%%%%%%
%%%%%%%%%%
\subsection{Displacement difficulty of fixed-height skew shapes}
\label{ssec:diff}

\begin{defn}
Let $\beta / \alpha$ be a fixed-height skew shape. The \emph{chain threshold} $\ct(\beta/\alpha)$ of $\beta / \alpha$ is the minimum $n$ such that there exists a sequence $\alpha = \gamma^0 < \gamma^1 < \cdots < \gamma^n = \beta$ of ramification sequences such that each $\gamma^n / \gamma^{n-1}$ is either a $1$-link or a $2$-link. The \emph{displacement difficulty} of $\beta / \alpha$ is the minimum number $\combdiff(\beta/\alpha)$ of $1$-links in such a sequence, i.e. $\combdiff(\beta / \alpha) = 2 \ct(\beta/\alpha) - |\beta / \alpha|.$

Define also the \emph{geometric difficulty} $\geodiff(\beta / \alpha)$ to be $2 \tg(\beta / \alpha) - | \beta / \alpha|$ (this may be negative).
\end{defn}

The chain threshold may be regarded as a version of threshold genus in which we consider limit linear series on chains of elliptic curves, rather than linear series on smooth curves.
An immediate consequence of the subadditivity is

\begin{cor}
For an fixed-height skew shape $\beta / \alpha$, $\tg(\beta / \alpha) \leq \ct(\beta / \alpha)$ and $\geodiff(\beta / \alpha) \leq \combdiff(\beta / \alpha).$
\end{cor}

\begin{lemma}
\label{lem:basicDeltaFacts}
For any $\alpha \leq \beta$ and integer $n$, $\combdiff\parenSkew{\beta}{\alpha} = \combdiff\parenSkew{n+\beta}{n+\alpha} = \combdiff\parenSkew{n-\alpha}{n-\beta}$, and for any $\alpha \leq \beta \leq \gamma$, $\combdiff(\gamma / \alpha) \leq \combdiff(\gamma/\beta) + \combdiff(\beta/\alpha)$. The same is true for $\geodiff$.
\end{lemma}

\begin{proof}
This follows from Proposition \ref{prop:tgProps}, Theorem \ref{thm:subadd}, and definitions.
\end{proof}

%%%%%%%%%%
%%%%%%%%%%
\subsection{Some useful $2$-links}
\label{ssec:usefulLinks}

The number $\combdiff$ can be computed algorithmically. The main impetus for this project is the following observation which, while vague, bears emphasis.

\begin{obs}
It is extremely common, when $|\beta / \alpha|$ is even, that $\combdiff(\beta / \alpha) = 0$.
\end{obs}

This subsection provides constructions of $2$-links, and therefore difficulty-$0$ skew shapes, useful in the proof of Theorem \ref{thm:main}.

\begin{rem}
Unfortunately, while it is very common that displacement difficulty is $0$, there are enough insidious exceptions that it is often maddeningly delicate to make general constructions. While working on this project, I tried several dozen constructions, many of which were found by computer search, before arriving at the choices below; the reader may be forgiven for not considering the specific choices below natural or obvious (but I hope they seem somewhat natural with the benefit of hindsight). The key observation was that it is useful to find ramification sequences $\alpha$ with the ``periodicity'' property $\combdiff\parenSkew{n + \alpha}{\alpha} = 0$ for some $n \geq 1$; such partitions were then found by computer searches, and are stated in Corollary \ref{cor:floorCeil} below. I mention this only because, when reading others' papers, I find myself perseverating often on how they could have naturally arrived at certain clever constructions and try to guess at the intuition concealed behind them. In this case, very little intuition was present in the author's mind; only patience and many discarded alternatives.
\end{rem}

\begin{defn}
For integers $n \geq a \geq b \geq c$, let $\tau^n_{a,b,c}$ denote the rank $n-1$ ramification sequence
$$\tau^n_{a,b,c} = 0^{n-a} 1^{a-b} 2^{b-c} 3^c.$$
\end{defn}

Visually, this ramification sequence has Young diagram consisting of three columns of height $a,b,c$.
The loose set is contained in
\begin{equation}
\label{eq:looseTau}
\{0,\ n-a,\ n-a+1,\ n-b+1,\ n-b+2,\ n-c+2,\ n-c+3,\ n+3 \}.
\end{equation}
The loose set need not include \emph{all} of these values; it does so only when $a,b,c,n$ are all distinct.

\begin{lemma}
\label{lem:2linksTau}
Fix integers $n \geq a \geq b \geq c$.
\begin{enumerate}
\item Suppose $c < b$ and $a < n$. If neither $n-a$ nor $c+1$ is divisible by $a-c+2$, then $\tau^n_{a+1,b,c+1} / \tau^n_{a,b,c}$ is a $2$-link.
\item Suppose $b < a < n$. If none of $n-a, b-c+1, b-c+2, b+2$ are divisible by $a-b+1$, then $\tau^n_{a+1,b+1,c} / \tau^n_{a,b,c}$ is a $2$-link.
\end{enumerate}
\end{lemma}

\begin{proof}
We apply Lemma \ref{lem:2links}. In the first part, the arithmetic progression $\Lambda$ is generated by $n-a$ and $n-c+2$ and has common difference $a-c+2$, and it suffices to check that $n-a, n-c+2$ are the only two elements of $\Lambda$ in list \eqref{eq:looseTau}. The three values $n-a+1,n-b+2,n-b+2$ cannot be in $\Lambda$ since they lie strictly between the adjacent elements $n-a$ and $n-c+2$, and $n-c+3$ cannot be in $\Lambda$ since the common difference is at least $2$. So it suffices that neither $0$ nor $n+3$ are in $\Lambda$, which amounts to the stated divisibility conditions. For the second part, the arithmetic progression is generated by $n-a$ and $n-b+1$; its common difference is $a-b+1$. Since this common difference is at least $2$, neither $n-a+1$ nor $n-b+2$ can be present, and it suffices to check that the other four elements in list \eqref{eq:looseTau} are not in $\Lambda$, which amounts to the stated divisibility conditions.
\end{proof}

The following two corollaries are illustrated in Figure \ref{fig:tauLinks}.

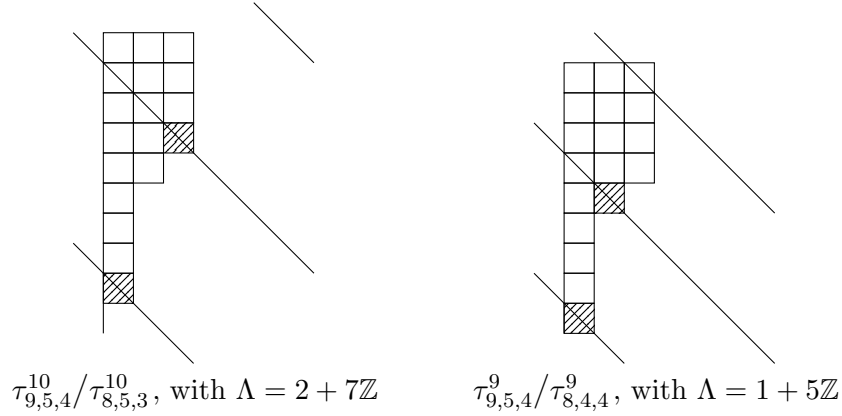
\begin{figure}

\begin{tabular}{ccc}
\begin{tikzpicture}
% Figure generated by calling dispDiagram([0, 0, 1, 1, 1, 2, 2, 3, 3, 3],[2, 9, 16],7) from figures.py
\draw (0.000, 0.000) -- (0.000, 0.400);
\draw (0.000, 0.400) -- (0.000, 0.800);
\draw (0.000,0.800) rectangle ++(0.400,0.400);
\draw (0.000,1.200) rectangle ++(0.400,0.400);
\draw (0.000,1.600) rectangle ++(0.400,0.400);
\draw (0.000,2.000) rectangle ++(0.400,0.400);
\draw (0.400,2.000) rectangle ++(0.400,0.400);
\draw (0.000,2.400) rectangle ++(0.400,0.400);
\draw (0.400,2.400) rectangle ++(0.400,0.400);
\draw (0.000,2.800) rectangle ++(0.400,0.400);
\draw (0.400,2.800) rectangle ++(0.400,0.400);
\draw (0.800,2.800) rectangle ++(0.400,0.400);
\draw (0.000,3.200) rectangle ++(0.400,0.400);
\draw (0.400,3.200) rectangle ++(0.400,0.400);
\draw (0.800,3.200) rectangle ++(0.400,0.400);
\draw (0.000,3.600) rectangle ++(0.400,0.400);
\draw (0.400,3.600) rectangle ++(0.400,0.400);
\draw (0.800,3.600) rectangle ++(0.400,0.400);
\draw (-0.400,1.200) -- (1.200,-0.400);
\draw[pattern=north east lines] (0.000,0.400) rectangle (0.400,0.800);
\draw (-0.400,4.000) -- (2.800,0.800);
\draw[pattern=north east lines] (0.800,2.400) rectangle (1.200,2.800);
\draw (2.000,4.400) -- (2.800,3.600);
\end{tikzpicture}
&&
\begin{tikzpicture}
% Figure generated by calling dispDiagram([0, 1, 1, 1, 1, 3, 3, 3, 3],[1, 6, 11],7) from figures.py
\draw (0.000, 0.000) -- (0.000, 0.400);
\draw (0.000,0.400) rectangle ++(0.400,0.400);
\draw (0.000,0.800) rectangle ++(0.400,0.400);
\draw (0.000,1.200) rectangle ++(0.400,0.400);
\draw (0.000,1.600) rectangle ++(0.400,0.400);
\draw (0.000,2.000) rectangle ++(0.400,0.400);
\draw (0.400,2.000) rectangle ++(0.400,0.400);
\draw (0.800,2.000) rectangle ++(0.400,0.400);
\draw (0.000,2.400) rectangle ++(0.400,0.400);
\draw (0.400,2.400) rectangle ++(0.400,0.400);
\draw (0.800,2.400) rectangle ++(0.400,0.400);
\draw (0.000,2.800) rectangle ++(0.400,0.400);
\draw (0.400,2.800) rectangle ++(0.400,0.400);
\draw (0.800,2.800) rectangle ++(0.400,0.400);
\draw (0.000,3.200) rectangle ++(0.400,0.400);
\draw (0.400,3.200) rectangle ++(0.400,0.400);
\draw (0.800,3.200) rectangle ++(0.400,0.400);
\draw (-0.400,0.800) -- (0.800,-0.400);
\draw[pattern=north east lines] (0.000,0.000) rectangle (0.400,0.400);
\draw (-0.400,2.800) -- (2.800,-0.400);
\draw[pattern=north east lines] (0.400,1.600) rectangle (0.800,2.000);
\draw (0.400,4.000) -- (2.800,1.600);

\end{tikzpicture}
\\
$\left. \tau^{10}_{9,5,4} \middle\slash \tau^{10}_{8,5,3} \right.$, with $\Lambda = 2 + 7 \ZZ$
&
\ \ \ \ \ 
&
$\left. \tau^9_{9,5,4} \middle\slash \tau^9_{8,4,4} \right.$, with $\Lambda = 1 + 5 \ZZ$
\end{tabular}

\caption{
Examples illustrating Corollaries \ref{cor:increaseAC} (left) and \ref{cor:increaseAB} (right). Note that on the right, the last diagonal line just barely misses an increasable element.
}
\label{fig:tauLinks}
\end{figure}

\begin{cor}
\label{cor:increaseAC}
If $n > a \geq b > c$ and $a-c \geq \lfloor \frac{n}{2} \rfloor \geq c$, then $\combdiff \parenSkew{ \tau^n_{a+1,b,c+1}}{\tau^n_{a,b,c}} = 0$.
\end{cor}

\begin{proof}
In Lemma \ref{lem:2linksTau}(1), we have $n-a \leq n-a + c \leq n - \floor{\frac{n}{2}} \leq \floor{\frac{n}{2}} + 1$ and $c+1 \leq \floor{\frac{n}{2}} + 1$. Since $a-c+2 \geq \floor{\frac{n}{2}}+2$, $a-c+2$ cannot divide either of these, so $\tau^n_{a+1,b,c+1} / \tau^n_{a,b,c}$ is a $2$-link.
\end{proof}

\begin{cor}
\label{cor:increaseAB}
For any odd integer $n \geq 5$,
$\combdiff \parenSkew{\tau^{n}_{n,\ceil{n/2},\floor{n/2}}}{\tau^{n}_{n-1,\floor{n/2},\floor{n/2}}} = 0$.
\end{cor}

\begin{proof}
Let $k = \floor{\frac{n}{2}}$; we must show that $\tau^{2k+1}_{2k+1, k+1,k} / \tau^{2k+1}_{2k,k,k}$ is a $2$-link. Using Lemma \ref{lem:2linksTau}(2), we must check that $k+1$ divides none of $\{1,2,k+2\}$, which follows from $k \geq 2$.
\end{proof}

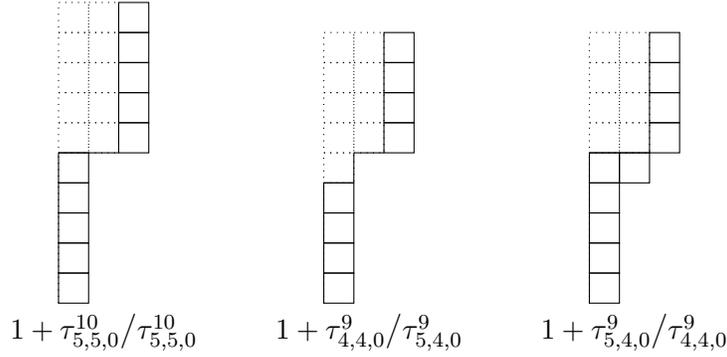
\begin{figure}

\begin{tabular}{ccccc}
\begin{tikzpicture}
% Figure generated by calling skewDottedBelow([0, 0, 0, 0, 0, 2, 2, 2, 2, 2],[1, 1, 1, 1, 1, 3, 3, 3, 3, 3],0.400) in figures.py
\draw[dotted] (0.000, 0.000) -- (0.000, 0.400);
\draw[dotted] (0.000, 0.400) -- (0.000, 0.800);
\draw[dotted] (0.000, 0.800) -- (0.000, 1.200);
\draw[dotted] (0.000, 1.200) -- (0.000, 1.600);
\draw[dotted] (0.000, 1.600) -- (0.000, 2.000);
\draw[dotted] (0.000,2.000) rectangle ++(0.400,0.400);
\draw[dotted] (0.400,2.000) rectangle ++(0.400,0.400);
\draw[dotted] (0.000,2.400) rectangle ++(0.400,0.400);
\draw[dotted] (0.400,2.400) rectangle ++(0.400,0.400);
\draw[dotted] (0.000,2.800) rectangle ++(0.400,0.400);
\draw[dotted] (0.400,2.800) rectangle ++(0.400,0.400);
\draw[dotted] (0.000,3.200) rectangle ++(0.400,0.400);
\draw[dotted] (0.400,3.200) rectangle ++(0.400,0.400);
\draw[dotted] (0.000,3.600) rectangle ++(0.400,0.400);
\draw[dotted] (0.400,3.600) rectangle ++(0.400,0.400);
\draw (0.000,0.000) rectangle ++(0.400,0.400);
\draw (0.000,0.400) rectangle ++(0.400,0.400);
\draw (0.000,0.800) rectangle ++(0.400,0.400);
\draw (0.000,1.200) rectangle ++(0.400,0.400);
\draw (0.000,1.600) rectangle ++(0.400,0.400);
\draw (0.000,2.000) -- (0.800, 2.000);
\draw (0.800,2.000) rectangle ++(0.400,0.400);
\draw (0.800,2.400) rectangle ++(0.400,0.400);
\draw (0.800,2.800) rectangle ++(0.400,0.400);
\draw (0.800,3.200) rectangle ++(0.400,0.400);
\draw (0.800,3.600) rectangle ++(0.400,0.400);

\end{tikzpicture}
&
\ \ \ 
&
\begin{tikzpicture}
% Figure generated by calling skewDottedBelow([0, 0, 0, 0, 1, 2, 2, 2, 2],[1, 1, 1, 1, 1, 3, 3, 3, 3],0.400) in figures.py
\draw[dotted] (0.000, 0.000) -- (0.000, 0.400);
\draw[dotted] (0.000, 0.400) -- (0.000, 0.800);
\draw[dotted] (0.000, 0.800) -- (0.000, 1.200);
\draw[dotted] (0.000, 1.200) -- (0.000, 1.600);
\draw[dotted] (0.000,1.600) rectangle ++(0.400,0.400);
\draw[dotted] (0.000,2.000) rectangle ++(0.400,0.400);
\draw[dotted] (0.400,2.000) rectangle ++(0.400,0.400);
\draw[dotted] (0.000,2.400) rectangle ++(0.400,0.400);
\draw[dotted] (0.400,2.400) rectangle ++(0.400,0.400);
\draw[dotted] (0.000,2.800) rectangle ++(0.400,0.400);
\draw[dotted] (0.400,2.800) rectangle ++(0.400,0.400);
\draw[dotted] (0.000,3.200) rectangle ++(0.400,0.400);
\draw[dotted] (0.400,3.200) rectangle ++(0.400,0.400);
\draw (0.000,0.000) rectangle ++(0.400,0.400);
\draw (0.000,0.400) rectangle ++(0.400,0.400);
\draw (0.000,0.800) rectangle ++(0.400,0.400);
\draw (0.000,1.200) rectangle ++(0.400,0.400);
\draw (0.000,1.600) -- (0.400, 1.600);
\draw (0.400, 1.600) -- (0.400, 2.000);
\draw (0.400,2.000) -- (0.800, 2.000);
\draw (0.800,2.000) rectangle ++(0.400,0.400);
\draw (0.800,2.400) rectangle ++(0.400,0.400);
\draw (0.800,2.800) rectangle ++(0.400,0.400);
\draw (0.800,3.200) rectangle ++(0.400,0.400);
\end{tikzpicture}
&
\ \ \
&
\begin{tikzpicture}
% Figure generated by calling skewDottedBelow([0, 0, 0, 0, 0, 2, 2, 2, 2],[1, 1, 1, 1, 2, 3, 3, 3, 3],0.400) in figures.py
\draw[dotted] (0.000, 0.000) -- (0.000, 0.400);
\draw[dotted] (0.000, 0.400) -- (0.000, 0.800);
\draw[dotted] (0.000, 0.800) -- (0.000, 1.200);
\draw[dotted] (0.000, 1.200) -- (0.000, 1.600);
\draw[dotted] (0.000, 1.600) -- (0.000, 2.000);
\draw[dotted] (0.000,2.000) rectangle ++(0.400,0.400);
\draw[dotted] (0.400,2.000) rectangle ++(0.400,0.400);
\draw[dotted] (0.000,2.400) rectangle ++(0.400,0.400);
\draw[dotted] (0.400,2.400) rectangle ++(0.400,0.400);
\draw[dotted] (0.000,2.800) rectangle ++(0.400,0.400);
\draw[dotted] (0.400,2.800) rectangle ++(0.400,0.400);
\draw[dotted] (0.000,3.200) rectangle ++(0.400,0.400);
\draw[dotted] (0.400,3.200) rectangle ++(0.400,0.400);
\draw (0.000,0.000) rectangle ++(0.400,0.400);
\draw (0.000,0.400) rectangle ++(0.400,0.400);
\draw (0.000,0.800) rectangle ++(0.400,0.400);
\draw (0.000,1.200) rectangle ++(0.400,0.400);
\draw (0.000,1.600) rectangle ++(0.400,0.400);
\draw (0.400,1.600) rectangle ++(0.400,0.400);
\draw (0.000,2.000) -- (0.800, 2.000);
\draw (0.800,2.000) rectangle ++(0.400,0.400);
\draw (0.800,2.400) rectangle ++(0.400,0.400);
\draw (0.800,2.800) rectangle ++(0.400,0.400);
\draw (0.800,3.200) rectangle ++(0.400,0.400);
\end{tikzpicture}
\\
$\left. 1 + \tau^{10}_{5,5,0} \middle\slash \tau^{10}_{5,5,0} \right.$
&&
$\left. 1 + \tau^9_{4,4,0} \middle\slash \tau^9_{5,4,0} \right.$
&&
$\left. 1 + \tau^{9}_{5,4,0} \middle\slash \tau^{9}_{4,4,0} \right.$
\end{tabular}

\caption{Examples of the difficulty-$0$ skew shapes in Corollary \ref{cor:floorCeil}, in both the even-height and odd-height case. For readability, in each skew shape $\beta / \alpha$, the boxes of the skew shape $\alpha / 0^n$ are shown with dotted lines.}
\label{fig:shiftTau}
\end{figure}

\begin{cor}
\label{cor:floorCeil}
For any integer $n \geq 4$,
$$\combdiff\parenSkew{1 + \tau^{n}_{\ceil{n/2},\floor{n/2},0}}{\tau^{n}_{\floor{n/2},\floor{n/2},0}} = \combdiff\parenSkew{1 + \tau^{n}_{\floor{n/2},\floor{n/2},0}}{\tau^{n}_{\ceil{n/2},\floor{n/2},0}} = 0.$$
See Figure \ref{fig:shiftTau} for an illustration.
\end{cor}

\begin{proof}
Let $m$ be either $\floor{\frac{n}{2}}$ or $\ceil{\frac{n}{2}}$. Corollary \ref{cor:increaseAC} shows, by a sequence of $\floor{\frac{n}{2}}$ $2$-links, that 
$$\combdiff\parenSkew{\tau^n_{m + \floor{n/2}, \floor{n/2}, \floor{n/2}}}{\tau^n_{m,\floor{n/2},0}} = 0.$$ 
If $n$ is even this gives both equations.  If $n$ is odd, then the two possibilities for $m$ give
$$\combdiff\parenSkew{\tau^{n}_{n-1,\floor{n/2},\floor{n/2}}}{\tau^{n}_{\floor{n/2},\floor{n/2},0}} = \combdiff\parenSkew{\tau^{n}_{n,\floor{n/2},\floor{n/2}}}{\tau^{n}_{\ceil{n/2},\floor{n/2},0}} = 0.$$
Combining with Corollary \ref{cor:increaseAB} then gives the desired equations.
\end{proof}

%%%%%%%%%%
%%%%%%%%%%
\section{Weierstrass points and twists of the canonical series}
\label{sec:weierstrass}

This section explains how to use a theorem of Komeda on dimensionally proper Weierstrass points to compute certain threshold genera. We return in this section to assuming $\operatorname{char} \FF = 0$.

A \emph{numerical semigroup} is a cofinite subset $S \subseteq \ZZ_{\geq 0}$ containing $0$ and closed under addition. The elements of $\ZZ_{\geq 0} \backslash S$ are called the \emph{gaps} of $S$, and the number of gaps is called the \emph{genus} of $S$. Denoting the gaps by $t_1 < t_2 < \cdots < t_g$, the \emph{weight} of $S$ is $\wt(S) = \sum_{n=1}^g(t_n-n)$. A numerical semigroup is called \emph{primitive} if, denoting by $s_1$ the first positive element of $S$, $2s_1 > t_g$. Primitivity has the useful consequence that decreasing the gaps preserves closure under addition. That is, if $0 < t_1' < \cdots < t_g'$ is any other increasing sequence of positive integers with $t_i' \leq t_i$ for all $1 \leq i \leq g$, then $S' = \ZZ_{\geq 0} \backslash \{t_1', \cdots, t_g'\}$ is also a primitive numerical semigroup.

A smooth once-marked curve $(C,q)$ of genus $g$ determines a numerical semigroup $S(C,q)$, called the \emph{Weierstrass semigroup}, which is the set of all pole orders at $q$ of regular functions on $C \backslash \{q\}$. Equivalently, the elements of $S(C,q)$ are $s_0 < s_1 < \cdots$, where $s_n = \min \{ s \in \ZZ: h^0(C, \cO(sq)) \geq n+1 \}$. Since $s_0 = 0$ and $s_n = n+g$ for $n \gg 0$, there are exactly $g$ gaps $t_1, \cdots, t_g$.

\begin{lemma}
\label{lem:ramCanonTwist}
For any $n \geq 1$, the vanishing sequence of $|\omega_C(nq)|$ at $q$ is
$a^{|\omega_C(nq)|}(q) = (0, 1, \cdots, n-2, n-1 + t_1, n-1+t_2, \cdots, n-1 + t_g),$
and the ramification sequence is therefore
$$\alpha^{|\omega_C(nq)|}(q) = (0, 0, \cdots, 0, t_1, t_2-1, \cdots, t_g-(g-1)).$$
\end{lemma}

\begin{proof}
Riemann-Roch shows that $m \geq 0$ is a vanishing order of $|\omega_C(nq)|$ if and only if $m+1-n$ is \emph{not} in the Weierstrass semigroup.
\end{proof}

Every numerical semigroup $S$ defines a locally closed locus (possibly empty) $\cM^S_{g,1} \subseteq \cM_{g,1}$, consisting of all marked curves $(C,q)$ such that $S(C,q) = S$. The local codimension of $\cM^S_{g,1}$ at any point is at most $\wt(S)$, and a point is called \emph{dimensionally proper} if equality holds.

\begin{thm}[\cite{kom91}]
\label{thm:komeda}
If $S$ is a primitive numerical semigroup of genus $g$, and $\wt(S) \leq g-1$, then $\cM^S_{g,1}$ has dimensionally proper points.
\end{thm}

\begin{rem}
Theorem \ref{thm:komeda} was originally proved with the bound $\wt(S) \leq g-2$ by Eisenbud and Harris \cite{eh87}, using limit linear series methods that inspired the techniques of this paper. There is a generalization allowing non-primitive semigroups in \cite{nonprimitive}, which works in characteristic $p$. Eisenbud and Harris made an argument by induction on $g$, but the inductive step failed for some specific semigroups of weight $g-1$. It was these that were analyzed by Komeda, and in fact it is precisely the weight $g-1$ semigroups analyzed by Komeda that we use in our proof of Theorem \ref{thm:main}.
\end{rem}

\begin{prop}
\label{prop:cMDP}
Let $S$ be a numerical semigroup of genus $g \geq 1$ with gaps $0 < t_1 < \cdots < t_g$, and $n$ a positive integer. Define a ramification sequence $\beta$ of rank $r = g-2+n$ by 
$$\beta = (0,\cdots,0,t_1, t_2-1,\cdots,t_g-(g-1)).$$
Then $\widetilde{\cG}^{\beta / 0^{r+1}}_g$ has dimensionally proper points if and only if $\cM^S_{g,1}$ has dimensionally proper points.
\end{prop}

\begin{proof}
Fix a twice-marked smooth curve $(C,p,q)$ of genus $g$. If $L = (\cL,V)$ is a linear series from $\widetilde{G}^{\beta / 0^{r+1}}(C,p,q)$, then since $\deg \cL = r+g \geq 2g-1$ and $\dim V = r+1 = \deg \cL + 1 -g = h^0(C,\cL)$, $L$ is necessarily complete. Since $\beta_{n-1} = 1$, it follows that $h^0(C, \cL(-nq)) = g$, so $h^1(C, \cL(-nq)) = 1$ and we must have $\cL(-nq) \cong \omega_C$, i.e. $\cL \cong \omega_C(nq)$. Together with Lemma \ref{lem:ramCanonTwist}, this shows that $\widetilde{G}^{\beta / 0^{r+1}}_g \to \cM_{g,2}$ is injective on points, and its image is precisely those $(C,p,q)$ for which $(C,q)$ is in $\cM^S_{g,1}$ and $p$ is unramified in $|\omega_C(nq)|$. The result now follows since $g - |\beta / 0^{r+1}| = - \wt(S)$.
\end{proof}

\begin{thm}
\label{thm:komedaDP}
If $S$ is a primitive semigroup with gaps $t_1, \cdots, t_g$ such that $\wt(S) \leq g-1$, $r$ is an integer with $r \geq g-1$, and $\beta = (0, \cdots, 0, t_1, t_2-1, \cdots, t_g-(g-1))$ where the number of $0$s is $r+1-g$ so that $\beta$ has rank $r$, then $\tg(\beta / 0^{r+1}) = g$.
\end{thm}

\begin{proof}
First, we claim $\cG^{\beta / 0^{r+1}}_h$ is empty for all $h<g$. Suppose that $C$ has genus $h$ and $G^{\beta / 0^{r+1}}(C)$ contains a point $L = (\cL,V)$. Let $\cL' = \cL(-(r-g+2)q))$; note $\deg \cL' = g+h-2$. Since $\beta_{r-g+1} = 1$, it follows that $h^0(C,\cL') \geq g$ and $h^1(C, \cL') \geq 1$, and thus $g+h-2 = \deg \cL' \leq 2h-2$, i.e. $h \geq g$, which shows the claim.

The theorem now follows from the claim: for all integers $h$ satisfying $g \leq h \leq | \beta / 0^{r+1}| = \wt(S) + g$, $\widetilde{\cG}^{\beta / 0^{r+1}}_h$ has dimensionally proper points.
The case $g = h$ follows from Proposition \ref{prop:cMDP} and Komeda's Theorem \ref{thm:komeda}. To deduce the result for $h>g$, we induct on the difference $h-g$, and use our analysis on $1$-links from the previous section.

Suppose that $g < h \leq \wt(S) + g$, and that the claim holds for smaller values of $h-g$. Let $t$ be the largest gap of $S$ such that $t-1 \in S$. Since $\wt(S) > 0$, $t > 1$. Let $S' = S \cup \{t\} \backslash \{t-1\}$. Since $S$ is a primitive semigroup, it follows that $S'$ is also a primitive semigroup. The weight of $S'$ is $\wt(S') = \wt(S)-1$. Let $\beta'$ be the rank-$r$ ramification sequence associated to $S'$. So $g \leq h-1 \leq |\beta' / 0^{r+1}|$; by inductive hypothesis, $\widetilde{\cG}^{\beta' / 0^{r+1}}_{h-1}$ has dimensionally proper points. Lemma \ref{lem:1links} shows that $\beta / \beta'$ is a $1$-link. Proposition \ref{prop:dpLinks} implies that $\widetilde{\cG}^{\beta / \beta'}_1$ has dimensionally proper points, and subadditivity implies that $\widetilde{\cG}^{\beta / 0^{r+1}}_h$ does as well. This completes the induction.
\end{proof}

%%%%%%%%%%
%%%%%%%%%%
\section{Proof of Theorem \ref{thm:main}}
\label{sec:proofMain}

We will deduce Theorem \ref{thm:main} from the following bound, claimed in Equation \eqref{eq:mainThmTG}.

\begin{thm}
\label{thm:main2}
For all integers $a,b \geq 2$, $\tg(a \times b) \leq \frac12 ab + 2.$
\end{thm}

\begin{proof}[Proof that Theorem \ref{thm:main2} implies Theorem \ref{thm:main}]
Suppose $g,r,d$ satisfy $r+1,g-d+r \geq 2$ and $\rho(g,r,d) \geq - g + 3$. Let $a = r+1$ and $b = g-d+r$. Then $g - ab \geq -g+3$, so $g \geq \ceil{\frac12 ab + \frac32} = \floor{\frac12 ab + 2}$. Assuming Theorem \ref{thm:main2}, $a,b \geq 2$ implies $\floor{ \frac12 ab + 2 } \geq \tg(a \times b)$, so $g \geq \tg((r+1) \times (g-d+r))$ and the result follows from the definition of threshold genus.
\end{proof}

The proof of Theorem \ref{thm:main2} occupies the rest of this section. First, note that if $a=2$ or $a=3$, the result follows from previously known results in Section \ref{sec:previous}, namely Equations \eqref{eq:a2} and \eqref{eq:maina3}; by symmetry of $\tg(a \times b)$ the result also follow if $b=2$ or $b=3$. So it suffices to consider $a,b \geq 4$. For the rest of the section, suppose we have fixed $a,b \geq 4$. Recall by Corollary \ref{cor:skewRect} that $\tg(a \times b) = \tg(b^a / 0^a)$; our strategy is to decompose the fixed-height skew shape $b^a / 0^a$ into four skew shapes, each of which has geometric difficulty at most $1$.

Define the following notation: let $k = \floor{\frac{a-1}{2}}$, so that either $a = 2k+1$ if $a$ is odd, or $a= 2k+2$ if $a$ is even. Choose any integer $\ell$ satisfying $1 \leq \ell \leq b-3$. 
In this notation, we break the skew shape $b^a / 0^a$ into the differences between the following four ramification sequences:
$$
0^a,\ 
 \tau^a_{2k+1,k,k},\ 
\ell + \tau^a_{\ceil{a/2}, \floor{a/2},0},\ 
b - \tau^a_{2k+1,k,k},\ 
b^a.$$
This sequence is illustrated in Figure \ref{fig:fourStage}. We consider each of the four skew shapes in turn.

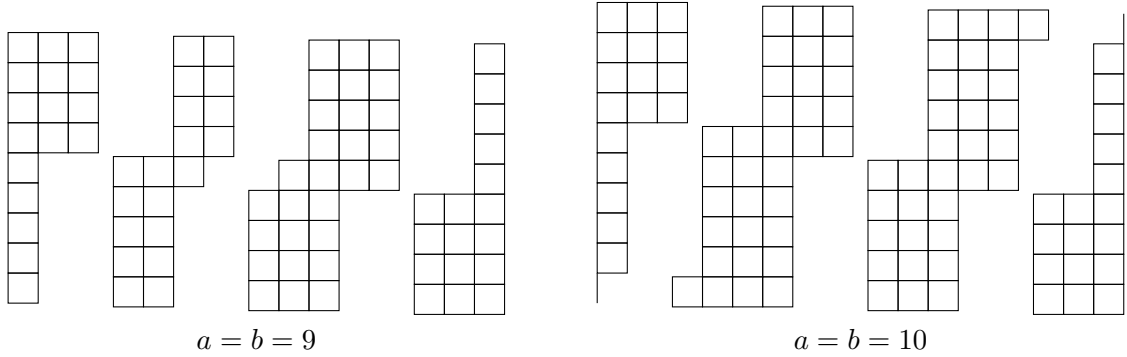
\begin{figure}

\begin{tabular}{ccc}

\begin{tikzpicture}
% Figure generated by calling fourStage(9,9,0.400,1.000,0.050) in figures.py.
\draw (0.000,0.000) rectangle ++(0.400,0.400);
\draw (0.000,0.400) rectangle ++(0.400,0.400);
\draw (0.000,0.800) rectangle ++(0.400,0.400);
\draw (0.000,1.200) rectangle ++(0.400,0.400);
\draw (0.000,1.600) rectangle ++(0.400,0.400);
\draw (0.000,2.000) rectangle ++(0.400,0.400);
\draw (0.400,2.000) rectangle ++(0.400,0.400);
\draw (0.800,2.000) rectangle ++(0.400,0.400);
\draw (0.000,2.400) rectangle ++(0.400,0.400);
\draw (0.400,2.400) rectangle ++(0.400,0.400);
\draw (0.800,2.400) rectangle ++(0.400,0.400);
\draw (0.000,2.800) rectangle ++(0.400,0.400);
\draw (0.400,2.800) rectangle ++(0.400,0.400);
\draw (0.800,2.800) rectangle ++(0.400,0.400);
\draw (0.000,3.200) rectangle ++(0.400,0.400);
\draw (0.400,3.200) rectangle ++(0.400,0.400);
\draw (0.800,3.200) rectangle ++(0.400,0.400);
\draw (1.400,-0.050) rectangle ++(0.400,0.400);
\draw (1.800,-0.050) rectangle ++(0.400,0.400);
\draw (1.400,0.350) rectangle ++(0.400,0.400);
\draw (1.800,0.350) rectangle ++(0.400,0.400);
\draw (1.400,0.750) rectangle ++(0.400,0.400);
\draw (1.800,0.750) rectangle ++(0.400,0.400);
\draw (1.400,1.150) rectangle ++(0.400,0.400);
\draw (1.800,1.150) rectangle ++(0.400,0.400);
\draw (1.400,1.550) rectangle ++(0.400,0.400);
\draw (1.800,1.550) rectangle ++(0.400,0.400);
\draw (2.200,1.550) rectangle ++(0.400,0.400);
\draw (2.200,1.950) rectangle ++(0.400,0.400);
\draw (2.600,1.950) rectangle ++(0.400,0.400);
\draw (2.200,2.350) rectangle ++(0.400,0.400);
\draw (2.600,2.350) rectangle ++(0.400,0.400);
\draw (2.200,2.750) rectangle ++(0.400,0.400);
\draw (2.600,2.750) rectangle ++(0.400,0.400);
\draw (2.200,3.150) rectangle ++(0.400,0.400);
\draw (2.600,3.150) rectangle ++(0.400,0.400);
\draw (3.200,-0.100) rectangle ++(0.400,0.400);
\draw (3.600,-0.100) rectangle ++(0.400,0.400);
\draw (4.000,-0.100) rectangle ++(0.400,0.400);
\draw (3.200,0.300) rectangle ++(0.400,0.400);
\draw (3.600,0.300) rectangle ++(0.400,0.400);
\draw (4.000,0.300) rectangle ++(0.400,0.400);
\draw (3.200,0.700) rectangle ++(0.400,0.400);
\draw (3.600,0.700) rectangle ++(0.400,0.400);
\draw (4.000,0.700) rectangle ++(0.400,0.400);
\draw (3.200,1.100) rectangle ++(0.400,0.400);
\draw (3.600,1.100) rectangle ++(0.400,0.400);
\draw (4.000,1.100) rectangle ++(0.400,0.400);
\draw (3.600,1.500) rectangle ++(0.400,0.400);
\draw (4.000,1.500) rectangle ++(0.400,0.400);
\draw (4.400,1.500) rectangle ++(0.400,0.400);
\draw (4.800,1.500) rectangle ++(0.400,0.400);
\draw (4.000,1.900) rectangle ++(0.400,0.400);
\draw (4.400,1.900) rectangle ++(0.400,0.400);
\draw (4.800,1.900) rectangle ++(0.400,0.400);
\draw (4.000,2.300) rectangle ++(0.400,0.400);
\draw (4.400,2.300) rectangle ++(0.400,0.400);
\draw (4.800,2.300) rectangle ++(0.400,0.400);
\draw (4.000,2.700) rectangle ++(0.400,0.400);
\draw (4.400,2.700) rectangle ++(0.400,0.400);
\draw (4.800,2.700) rectangle ++(0.400,0.400);
\draw (4.000,3.100) rectangle ++(0.400,0.400);
\draw (4.400,3.100) rectangle ++(0.400,0.400);
\draw (4.800,3.100) rectangle ++(0.400,0.400);
\draw (5.400,-0.150) rectangle ++(0.400,0.400);
\draw (5.800,-0.150) rectangle ++(0.400,0.400);
\draw (6.200,-0.150) rectangle ++(0.400,0.400);
\draw (5.400,0.250) rectangle ++(0.400,0.400);
\draw (5.800,0.250) rectangle ++(0.400,0.400);
\draw (6.200,0.250) rectangle ++(0.400,0.400);
\draw (5.400,0.650) rectangle ++(0.400,0.400);
\draw (5.800,0.650) rectangle ++(0.400,0.400);
\draw (6.200,0.650) rectangle ++(0.400,0.400);
\draw (5.400,1.050) rectangle ++(0.400,0.400);
\draw (5.800,1.050) rectangle ++(0.400,0.400);
\draw (6.200,1.050) rectangle ++(0.400,0.400);
\draw (6.200,1.450) rectangle ++(0.400,0.400);
\draw (6.200,1.850) rectangle ++(0.400,0.400);
\draw (6.200,2.250) rectangle ++(0.400,0.400);
\draw (6.200,2.650) rectangle ++(0.400,0.400);
\draw (6.200,3.050) rectangle ++(0.400,0.400);
\end{tikzpicture}
&\ \ \ \ \ &
\begin{tikzpicture}
% Figure generated by calling fourStage(10,10,0.400,1.000,0.050) in figures.py.
\draw (0.000, 0.000) -- (0.000, 0.400);
\draw (0.000,0.400) rectangle ++(0.400,0.400);
\draw (0.000,0.800) rectangle ++(0.400,0.400);
\draw (0.000,1.200) rectangle ++(0.400,0.400);
\draw (0.000,1.600) rectangle ++(0.400,0.400);
\draw (0.000,2.000) rectangle ++(0.400,0.400);
\draw (0.000,2.400) rectangle ++(0.400,0.400);
\draw (0.400,2.400) rectangle ++(0.400,0.400);
\draw (0.800,2.400) rectangle ++(0.400,0.400);
\draw (0.000,2.800) rectangle ++(0.400,0.400);
\draw (0.400,2.800) rectangle ++(0.400,0.400);
\draw (0.800,2.800) rectangle ++(0.400,0.400);
\draw (0.000,3.200) rectangle ++(0.400,0.400);
\draw (0.400,3.200) rectangle ++(0.400,0.400);
\draw (0.800,3.200) rectangle ++(0.400,0.400);
\draw (0.000,3.600) rectangle ++(0.400,0.400);
\draw (0.400,3.600) rectangle ++(0.400,0.400);
\draw (0.800,3.600) rectangle ++(0.400,0.400);
\draw (1.000,-0.050) rectangle ++(0.400,0.400);
\draw (1.400,-0.050) rectangle ++(0.400,0.400);
\draw (1.800,-0.050) rectangle ++(0.400,0.400);
\draw (2.200,-0.050) rectangle ++(0.400,0.400);
\draw (1.400,0.350) rectangle ++(0.400,0.400);
\draw (1.800,0.350) rectangle ++(0.400,0.400);
\draw (2.200,0.350) rectangle ++(0.400,0.400);
\draw (1.400,0.750) rectangle ++(0.400,0.400);
\draw (1.800,0.750) rectangle ++(0.400,0.400);
\draw (2.200,0.750) rectangle ++(0.400,0.400);
\draw (1.400,1.150) rectangle ++(0.400,0.400);
\draw (1.800,1.150) rectangle ++(0.400,0.400);
\draw (2.200,1.150) rectangle ++(0.400,0.400);
\draw (1.400,1.550) rectangle ++(0.400,0.400);
\draw (1.800,1.550) rectangle ++(0.400,0.400);
\draw (2.200,1.550) rectangle ++(0.400,0.400);
\draw (1.400,1.950) rectangle ++(0.400,0.400);
\draw (1.800,1.950) rectangle ++(0.400,0.400);
\draw (2.200,1.950) rectangle ++(0.400,0.400);
\draw (2.600,1.950) rectangle ++(0.400,0.400);
\draw (3.000,1.950) rectangle ++(0.400,0.400);
\draw (2.200,2.350) rectangle ++(0.400,0.400);
\draw (2.600,2.350) rectangle ++(0.400,0.400);
\draw (3.000,2.350) rectangle ++(0.400,0.400);
\draw (2.200,2.750) rectangle ++(0.400,0.400);
\draw (2.600,2.750) rectangle ++(0.400,0.400);
\draw (3.000,2.750) rectangle ++(0.400,0.400);
\draw (2.200,3.150) rectangle ++(0.400,0.400);
\draw (2.600,3.150) rectangle ++(0.400,0.400);
\draw (3.000,3.150) rectangle ++(0.400,0.400);
\draw (2.200,3.550) rectangle ++(0.400,0.400);
\draw (2.600,3.550) rectangle ++(0.400,0.400);
\draw (3.000,3.550) rectangle ++(0.400,0.400);
\draw (3.600,-0.100) rectangle ++(0.400,0.400);
\draw (4.000,-0.100) rectangle ++(0.400,0.400);
\draw (4.400,-0.100) rectangle ++(0.400,0.400);
\draw (3.600,0.300) rectangle ++(0.400,0.400);
\draw (4.000,0.300) rectangle ++(0.400,0.400);
\draw (4.400,0.300) rectangle ++(0.400,0.400);
\draw (3.600,0.700) rectangle ++(0.400,0.400);
\draw (4.000,0.700) rectangle ++(0.400,0.400);
\draw (4.400,0.700) rectangle ++(0.400,0.400);
\draw (3.600,1.100) rectangle ++(0.400,0.400);
\draw (4.000,1.100) rectangle ++(0.400,0.400);
\draw (4.400,1.100) rectangle ++(0.400,0.400);
\draw (3.600,1.500) rectangle ++(0.400,0.400);
\draw (4.000,1.500) rectangle ++(0.400,0.400);
\draw (4.400,1.500) rectangle ++(0.400,0.400);
\draw (4.800,1.500) rectangle ++(0.400,0.400);
\draw (5.200,1.500) rectangle ++(0.400,0.400);
\draw (4.400,1.900) rectangle ++(0.400,0.400);
\draw (4.800,1.900) rectangle ++(0.400,0.400);
\draw (5.200,1.900) rectangle ++(0.400,0.400);
\draw (4.400,2.300) rectangle ++(0.400,0.400);
\draw (4.800,2.300) rectangle ++(0.400,0.400);
\draw (5.200,2.300) rectangle ++(0.400,0.400);
\draw (4.400,2.700) rectangle ++(0.400,0.400);
\draw (4.800,2.700) rectangle ++(0.400,0.400);
\draw (5.200,2.700) rectangle ++(0.400,0.400);
\draw (4.400,3.100) rectangle ++(0.400,0.400);
\draw (4.800,3.100) rectangle ++(0.400,0.400);
\draw (5.200,3.100) rectangle ++(0.400,0.400);
\draw (4.400,3.500) rectangle ++(0.400,0.400);
\draw (4.800,3.500) rectangle ++(0.400,0.400);
\draw (5.200,3.500) rectangle ++(0.400,0.400);
\draw (5.600,3.500) rectangle ++(0.400,0.400);
\draw (5.800,-0.150) rectangle ++(0.400,0.400);
\draw (6.200,-0.150) rectangle ++(0.400,0.400);
\draw (6.600,-0.150) rectangle ++(0.400,0.400);
\draw (5.800,0.250) rectangle ++(0.400,0.400);
\draw (6.200,0.250) rectangle ++(0.400,0.400);
\draw (6.600,0.250) rectangle ++(0.400,0.400);
\draw (5.800,0.650) rectangle ++(0.400,0.400);
\draw (6.200,0.650) rectangle ++(0.400,0.400);
\draw (6.600,0.650) rectangle ++(0.400,0.400);
\draw (5.800,1.050) rectangle ++(0.400,0.400);
\draw (6.200,1.050) rectangle ++(0.400,0.400);
\draw (6.600,1.050) rectangle ++(0.400,0.400);
\draw (6.600,1.450) rectangle ++(0.400,0.400);
\draw (6.600,1.850) rectangle ++(0.400,0.400);
\draw (6.600,2.250) rectangle ++(0.400,0.400);
\draw (6.600,2.650) rectangle ++(0.400,0.400);
\draw (6.600,3.050) rectangle ++(0.400,0.400);
\draw (7.000, 3.450) -- (7.000, 3.850);
\end{tikzpicture}
\\
$ a = b = 9$
&&
$ a = b = 10$
\end{tabular}

\caption{Two examples of the sequence of four skew shapes in the Proof of Theorem \ref{thm:main2}. Each skew shape has geometric difficulty at most $1$.}
\label{fig:fourStage}
\end{figure}

\begin{lemma}
\label{lem:stageOne}
$\geodiff\parenSkew{\tau^{a}_{2k+1,k,k}}{0^a} = 1$.
\end{lemma}

\begin{proof}
Let $S = \{0, k+2, k+3 \} \cup \{n \in \ZZ: n \geq 2k+4\}$. 
%The minimum positive element of $S$ is $s_1 = k+2$, while the largest gap of $S$ is $2k+3 < 2(k+2)$, so $S$ is a primitive numerical semigroup. Its gaps are $t_n = n$ for $1 \leq n \leq k+1$, $t_n = n+2$ for $k+2 \leq n \leq 2k+1$, so the genus of $S$ is $g = 2k+1$ and its weight is $(k+2)\cdot 0 + k \cdot 2 = 2k = g-1$. 
This is a primitive numerical semigroup of genus $2k+1 = g$ and weight $2k = g-1$.
So we can apply Theorem \ref{thm:komedaDP} to $S$, with $r = a-1$. The ramification sequence $\beta$ is
$\beta = 0^{a-g} 1^{k+1} 3^k = \tau^{a}_{2k+1, k, k}$, so $\tg\parenSkew{\tau^{a}_{2k+1,k,k}}{0^a} = g $. Since $| \tau^a_{2k+1,k,k} / 0^a | = 4k+1 = 2g-1$, it follows that $\geodiff\parenSkew{\tau^{a}_{2k+1,k,k}}{0^a} = 2g - (2g-1) = 1$.
\end{proof}

\begin{lemma}
\label{lem:stageTwo}
$\geodiff\parenSkew{\ell + \tau^{a}_{\ceil{a/2}, \floor{a/2}, 0}}{\tau^a_{2k+1,k,k}}  \leq 1$. 
\end{lemma}

\begin{proof}
First, note that we may apply Corollary \ref{cor:floorCeil} $\ell-1$ times in a row, using Lemma \ref{lem:basicDeltaFacts} to handle the added constants, to deduce that
\begin{equation}
\label{eq:inductFloorCeil}
\combdiff\parenSkew{\ell + \tau^a_{ m, \floor{a/2}, 0}}{1 + \tau^a_{\floor{a/2}, \floor{a/2}, 0}} = 0,
\mbox{ where } m = \begin{cases}
\floor{\frac{a}{2}} & \mbox{ if } \ell \mbox{ is odd, }\\
\ceil{\frac{a}{2}} & \mbox{ if } \ell \mbox{ is even. }
\end{cases}
\end{equation}

{\it Case 1:} $a$ is odd. Then $\tau^a_{2k+1,k,k} = 1+ \tau^a_{\floor{a/2}, \floor{a/2},0}$. If $\ell$ is even the result follows immediately from Equation \eqref{eq:inductFloorCeil}, while if $\ell$ is odd then it follows from Equation \eqref{eq:inductFloorCeil} and the fact that $\tau^a_{\ceil{a/2},\floor{a/2},0} / \tau^a_{\floor{a/2},\floor{a/2},0}$ is a $1$-link.

{\it Case 2:} $a$ is even. Then $a = 2k+2$. By Equation \eqref{eq:inductFloorCeil}, it suffices to show that
$$\combdiff\parenSkew{1+\tau^a_{\floor{a/2},\floor{a/2}},0}{\tau^a_{2k+1,k,k}} \leq 1.$$
This can be seen by noting that $\tau^a_{2k+1,k+1,k} / \tau^a_{2k+1,k,k}$ is a $1$-link, $\tau^a_{2k+2,k+1,k+1} / \tau^a_{2k+1,k+1,k}$ is a $2$-link by Corollary \ref{cor:increaseAC}, and $\tau^a_{2k+2,k+1,k+1} = \tau^a_{a, \floor{a/2}, \floor{a/2}} = 1 + \tau^a_{\floor{a/2},\floor{a/2},0}$.
\end{proof}

We are halfway there, and the other half is just a matter of doing the same steps backwards. 

\begin{lemma}
\label{lem:stagesThreeFour}
$\combdiff \left(b - \tau^a_{2k+1,k,k} \middle\slash  \ell + \tau^a_{\ceil{a/2},  \floor{a/2},0}  \right) \leq 1$ 
and
$\geodiff\parenSkew{b^a}{b - \tau^a_{2k+1,k,k}} \leq 1$.
\end{lemma}

\begin{proof}
Subtracting the ramification sequences from $b$, the first inequality is equivalent to
$$\combdiff\parenSkew{(b-\ell-2) + \tau^a_{\ceil{a/2}, \floor{a/2},0}}{\tau^a_{2k+1,k,k}} \leq 1$$
by Lemma \ref{lem:basicDeltaFacts}. Since $b - \ell -2 \geq 1$, this follows from Lemma \ref{lem:stageTwo} (with $\ell$ replaced by $b-\ell-2$).
Similarly, the second inequality follows from Lemma \ref{lem:stageOne}.
\end{proof}

Lemmas \ref{lem:stageOne}, \ref{lem:stageTwo}, \ref{lem:stagesThreeFour} and subadditivity of $\geodiff$ now combine to show that $\geodiff(b^a / 0^a) \leq 4$. Equivalently, $\tg(b^a / 0^a) \leq \frac12 ab + 2$. This completes the proof of Theorem \ref{thm:main2}, and thus Theorem \ref{thm:main}.

%%%%%%%%%%
%%%%%%%%%%
\section{Proof of Theorem \ref{thm:asy}}
\label{sec:proofAsy}

We will prove Theorem \ref{thm:asy} by way of a bound on $\tg(a \times b)$, whose asymptotic form was claimed in Equation \eqref{eq:asyThmTG}. To formulate the bound, recall the function from Pareschi's results (Section \ref{ssec:a4}).

$$
f(d) = \frac{1}{1536} (32d-215)^{3/2} + \frac{23}{16}d - \frac{541}{1536} (32d-215)^{1/2} - \frac{397}{128}
$$

Define from this the following function, with asymptotic $\cO(b^{2/3})$.
$$
h(b) = -3 + \min \{ d \geq 20: f(d) -d + 3 \geq b \}
$$

\begin{lemma}
\label{lem:asyPrecise}
For all positive integers $a,b$,
$$\tg(a \times b) \leq \ceil{ \frac{a}{4} } b + \floor{ \frac{a}{4}} h(b) + 
\begin{cases}
1 & \mbox{ if } a \equiv 2 \pmod{4},\\
\ceil{ \frac12  + \sqrt{ 2b + \frac14}\ } & \mbox{ if } a \equiv 3 \pmod{4},\\
0 & \mbox{ otherwise.}
\end{cases}$$
\end{lemma}

\begin{proof}
Let $r \equiv a \pmod{4}$ denote the residue of $a$ in $\{0,1,2,3\}$. 
By subadditivity, 
$$\tg(a \times b) \leq \floor{ \frac{a}{4}} \tg(4 \times b) + \tg(r \times b),$$
where we understand $\tg(0 \times b)$ to be $0$. By Equation \eqref{eq:tg4bexplicit}, $\tg(4 \times b) \leq b + h(b)$, and by Equations \eqref{eq:a1}, \eqref{eq:a2}, \eqref{eq:a3}, $\tg(r \times b)$ is  $b$ if $r=1$, at most $b+1$ if $r=2$, and at most $b + \ceil{ \frac12  + \sqrt{ 2b + \frac14}\ }$ if $r=3$; the result follows.
\end{proof}

\begin{rem}
\label{rem:comparison}
Although we are primarily interested in the asymptotics of this bound, the precise bound is easy to compute; it is just messy to write down by hand. The explicit formula also makes it possible to compare directly to the bound $\tg(a \times b) \leq \frac12 ab + 2$ obtained in Theorem \ref{thm:main2}. For example, consider $a=b=61$. In this case, $\frac12 ab +2 = 1862.5$, so Theorem \ref{thm:main2} gives $\tg(61 \times 61) \leq 1862$. However, the bound in Lemma \ref{lem:asyPrecise} is $1876$. Therefore, for curves with genus as high as $1875$, there are choices of $g,r,d$ for which Theorem \ref{thm:main} guarantees existence of dimensionally proper points of $\cG^r_{g,d}$ but the explicit bound underlying Theorem \ref{thm:asy} does not.
\end{rem}

In the Theorem below, recall the notation $\eps_a = \frac{(-a) \operatorname{mod} 4}{\ceil{a/4}} = \frac{ 4 \ceil{a/4} - a}{\ceil{a/4}}$.

\begin{thm}
\label{thm:asy2}
For all integers $b \geq a \geq 1$, $\tg(a \times b) = \ceil{\frac{a}{4}} b + \cO((ab)^{5/6}) = \frac{1}{4 - \eps_a} ab + \cO((ab)^{5/6})$.
\end{thm}

\begin{proof}
Lemma \ref{lem:asyPrecise} and the fact that $h(b) = \cO(b^{2/3})$ shows that $\tg(a \times b) \leq \ceil{\frac{a}{4}} b + \cO( a b^{2/3}) + \cO(b^{1/2})$. Using the assumption $a \leq b$, we have $a b^{2/3} \leq a^{5/6} b^{5/6}$, and of course $b^{1/2} \leq a^{5/6} b^{5/6}$, so these last two error bounds may be combined to $\cO( (ab)^{5/6})$. The second equation follows since $\ceil{\frac{a}{4}}  = \frac{a}{4-\eps_a}$.
\end{proof}

\begin{proof}[Proof of Theorem \ref{thm:asy}]
Choose a constant $C$ such that $\tg(a \times b) \leq \frac{1}{4-\eps_a} ab + \frac{1}{16} C (ab)^{5/6}$ for all $b \geq a \geq 1$; such a constant exists by Theorem \ref{thm:asy2}. Suppose $g,r,d$ are nonnegative integers with $d \leq g-1$ and $0 \geq \rho(g,r,d) \geq (-3+\eps_{r+1}) g + C g^{5/6}$. Define $a=r+1$ and $b = g-d+r$. Our assumptions imply that $a \leq b$, $g \leq ab$, and $(4-\eps_a) g \geq ab +  C g^{5/6}$. Note that $\eps_a < 4$; dividing by $4-\eps_a$ gives
$g \geq \frac{1}{4-\eps_a} ab + \frac{C}{4-\eps_a} g^{5/6} \geq \frac{1}{4-\eps_a}ab + \frac{C}{4} g^{5/6}$. This implies in part that $g \geq \frac{1}{4-\eps_a} ab \geq \frac14 ab$, and therefore $\frac{C}{4} g^{5/6} \geq \frac{C}{4^{11/6}} (ab)^{5/6} \geq \frac{1}{16}C (ab)^{5/6}$, and therefore
$g \geq \frac{1}{4-\eps_a} ab + \frac{1}{16} C (ab)^{5/6}$.
By choice of $C$, this means that $g \geq \tg(a \times b)$, so $\cG^r_{g,d}$ has dimensionally proper points. 
\end{proof}

\appendix

\section{Connections to Hurwitz--Brill--Noether theory}\label{app:hbn}

The emerging subject of \emph{Hurwitz--Brill--Noether theory} provides another source of components of $\cG^r_{g,d}$ in the $\rho(g,r,d) < 0$ range, which are typically not dimensionally proper, and also suggests an analog of the main question of this paper. A comprehensive treatment of Hurwitz--Brill--Noether theory may be found in \cite{larsonLarsonVogt}, where the reader can find details of most of what is stated here. The methods and philosophy of Hurwitz--Brill--Noether theory are similar to those of this paper, and indeed have a common origin. This appendix briefly explains the connection, and states an analog of Question \ref{qu:underwaterBN}.

Whereas Brill--Noether theory studies general points in $\cM_g$, Hurwitz--Brill--Noether theory considers general curves of fixed gonality $k$, or alternatively general points in the Hurwitz space $\cH_{k,g}$.

\subsection{A source of non-dimensionally proper components}
A conjecture from \cite{pflKgonal}, proved in \cite{jensenRanganathan}, states that, if $C$ is a general curve of gonality $k$, then the largest irreducible component of $G^r_d(C)$ has the following dimension when it is nonnegative. Otherwise $G^r_d(C)$ is empty.
\begin{equation}
\rho_k(g,r,d) = \max \Big\{ \rho(g,r-\ell,d) - \ell k:\ 0 \leq \ell \leq \min \{r, g-d+r-1\} \Big\}
\end{equation}
Later results \cite{cpjComponents, cpjMethods, larsonHBN} refined this picture by identifying which \emph{other} dimensions of components occur besides this maximum. Interestingly, all such dimensions are of the form $\rho(g,r-\ell,d) - \ell k$ for various values of $\ell$. With hindsight, the proper way of understanding this phenomenon is to study not Brill--Noether loci themselves but \emph{splitting-type loci}, which we define in Section \ref{ssec:hbnRefined}.

In particular, one obtains in this way an irreducible locus $X \subseteq \cG^r_{g,d}$, whose image in $\cM_g$ is precisely the image of $\cG^1_{g,k}$, of relative dimension $\rho(g,1,k) + \rho_k(g,r,d)$. 
Of course, $X$ may not be a component of $\cG^r_{g,d}$, since it may spread beyond the $k$-gonal locus, but this would only make this relative dimension larger.
This relative dimension may exceed $\rho(g,r,d)$, and one obtains non-dimensionally proper components.
Conceptually, one may distill the slogan: \emph{Brill--Noether loci in $\cM_g$ interact.}  
Here, by a ``Brill--Noether locus in $\cM_g$'' we mean the image of some $\cG^r_{g,d} \to \cM_g$.
%Once a curve is known to have one exceptional type of linear series, this causes ripple effects that cause \emph{other} Brill--Noether loci to be larger than expected, resulting in non-dimensionally proper components of $\cG^r_{g,d}$.

\begin{rem}
The existence of such loci supported over the $k$-gonal locus, has recently been useful in the classification of maximal Brill--Noether loci \cite{auelHaburcak}.
\end{rem}

\subsection{A common origin}
This paper and recent work on Hurwitz--Brill--Noether theory have a common origin in the combinatorics of displacement of partitions. This in turn is closely related to the combinatorics of $k$-core tableaux, as described e.g. in \cite{lapointeMorse}. The operations $\disp^+_\Lambda$ and $\disp^-_\Lambda$ discussed in Section \ref{sec:difficulty} may be used to calculate the dimensions of spaces of limit linear series with imposed ramification on chains of genus $1$ curves (this paper simplifies exposition by only attaching one elliptic curve at time and smoothing at each step, but we could just as well have worked with chains). They may also be used for an analogous dimension calculation on a chain of genus $1$ \emph{tropical} curves, as in \cite{pflChains}. The resulting dimension depends on the orders of torsion between the attachment points in the chain, as well as the marked points at the ends of the chain.

The proof of Theorem \ref{thm:main} in the present paper makes use of chains with varying torsion order, chosen carefully to make regeneration possible. In contrast, many recent papers on Hurwitz--Brill--Noether theory \cite{pflKgonal, jensenRanganathan, cpjComponents, cpjMethods, larsonHBN, larsonLarsonVogt}  use chains in which all orders of torsion are $k$. This is because these chains arise as specializations of $k$-gonal curves, with two marked points of total ramification. For example, in \cite{pflKgonal}, the expected dimension $\rho_k(g,r,d)$ was derived in terms of displacement as follows. Call the operation $\disp^+_\Lambda$, where $\Lambda$ is an arithmetic progression of common difference $k$, a \emph{displacement modulo $k$}. The expected \emph{co}-dimension $u = g - \rho_k(g,r,d)$ is equal to the minimum number of displacements modulo $k$ needed to obtain, starting from the empty partition, a partition containing the $(r+1) \times (g-d+r)$ rectangle. This minimum may be found by finding the minimum number of symbols in a type of tableau called a \emph{$k$-uniform displacement tableau} in \cite{pflKgonal} or a \emph{$k$-regular tableau} in \cite{larsonLarsonVogt} and elsewhere.

\subsection{The challenge of regeneration}
A crucial difficulty in Hurwitz--Brill--Noether theory is the issue of regeneration of linear series from a chain of $k$-torsion genus $1$ curves to a smooth genus $g$ curve. Such regeneration was not attempted in \cite{pflKgonal}, which is why only an upper bound on dimension could be proved there. The smoothing techniques used in this paper (Section \ref{ssec:regen}) are not applicable, precisely because the linear series to be constructed are not dimensionally proper. The basic issue is that the naive dimension estimate used in the regeneration theorem of limit linear series does not account for additional structure imposed by the existence of a line bundle in $G^1_k(C)$.
This difficulty has now been solved in three distinct ways: using logarithmic deformation theory of tropical scrolls \cite{jensenRanganathan}, using deformation theory of splitting loci \cite{larsonHBN}, and using the notion of $\vec{e}$-nested linear series \cite{larsonLarsonVogt}. Variants of these techniques may be useful in addressing Question \ref{qu:underwaterHBN}, formulated at the end of this appendix.

The reason why the usual regeneration theorem for limit linear series is not useful for Hurwitz--Brill--Noether theory is neatly encapsulated by a concept from combinatorics. In the regeneration theorem, one can measure ramification of a linear series with a partition or ramification sequence $\lambda$, and the size $| \lambda |$ serves as a codimension estimate. In a certain sense, each box of the Young diagram of $\lambda$ contributes one equation. But in Hurwitz--Brill--Noether theory, the partitions that arise are special partitions called $k$-core partitions, many of the equations become redundant, and a lower codimension estimate should be used.

\subsection{The role of $k$-core partitions}
In fact, the partitions that can be obtained by a sequence of displacements modulo $k$ are well-studied and have many deep combinatorial properties. They are called \emph{$k$-core partitions}; the combinatorics of such partitions was later used to obtain much more precise results via tropical methods in \cite{cpjComponents, cpjMethods}, and plays a crucial role in the regeneration techniques developed in \cite{larsonLarsonVogt}. For example, the multiple components of $G^r_d(C)$, for $C$ a general $k$-gonal curve, are classified by the minimal $k$-core partitions that contain a  given rectangle.

A $k$-core partition may be defined as a partition with no hooks of length $k$, or equivalently no hook lengths divisible by $k$. 
A crucial fact about $k$-core partitions, as compared with arbitrary partitions, is that if a $k$-core partition $\lambda$ is obtained from the empty partition by a minimal sequence of displacements modulo $k$, then the length of this minimal sequence does not depend on the sequence chosen. We will call this number the \emph{length} of the $k$-core, and denote it $|\lambda|_k$. It is the proper replacement for the more naive $|\lambda|$ that occurs in the codimension estimate from limit linear series.

The length $|\lambda|_k$ has another convenient description: it is the number of boxes in the Young diagram of the partition whose hook length is less than $k$ \cite[Lemma 31]{lapointeMorse}. The relationship between length and displacement is neatly encapsulated by the equation (which follows, in different notation, from \cite[Proposition 22]{lapointeMorse})
\begin{equation}
\left| \disp^+_{\Lambda} \lambda \right|_k = \begin{cases}
\left| \lambda \right|_k + 1 & \mbox{ if } \disp^+_{\Lambda} \lambda \neq \lambda\\
\left| \lambda \right|_k & \mbox{ otherwise,}
\end{cases}
\end{equation} 
where $\Lambda$ is any arithmetic progression with common difference $k$, and $\lambda$ is a $k$-core.

\begin{rem} \label{rem:whyKCores}
The reason $k$-cores are the partitions that arise naturally in Hurwitz--Brill--Noether theory, and why the length of a $k$-core is more natural than the size of the partition when making codimension estimates, is most transparent in the following special situation. Let $f: C \to \PP^1$ be a degree $k$ genus $g$ cover, and suppose that $f$ has a point $p \in C$ of total ramification. Assume that $f(p)  = \infty$, so that $f$ is a rational function with pole divisor $kp$. The assumption that the gonality pencil has one or two points of total ramification is quite convenient, and is used extensively in \cite{larsonLarsonVogt}. Fortunately, covers with two points of total ramification appear to be ``general enough'' to satisfy the main theorems of Hurwitz--Brill--Noether theory.

We claim that for any line bundle $\cL$, the ramification sequence $\alpha = (\alpha_0, \cdots, \alpha_r)$ of the complete linear series $L = (\cL, H^0(C,\cL))$ is a $k$-core partition. Here we abuse notation slightly and regard a ``partition'' as a multiset of \emph{nonnegative} integers, and regard two partitions as ``the same'' if one is obtained by adding some number of $0$s to the other.

Here is a sketch of the proof; the reader is encouraged to draw some pictures to convince themself (see for example the pictures in \cite[\S 5.1]{larsonLarsonVogt}; their ``vertical and horizontal line segements'' correspond to integers that are, or are not, vanishing orders).  The boxes of the Young diagram of $\alpha$ are in bijection with pairs $(a,b)$ of nonnegative integers, where $a$ is in the vanishing sequence, $b$ is \emph{not} in the vanishing sequence, and $a > b$. The hook length of this box is the difference $a-b$. 

Now, the fact that $\alpha$ is a $k$-core follows from the observation that, if $a \geq k$ is a vanishing order, then $a-k$ is also a vanishing order; this is because we can multiply sections by $f$. So in every pair $(a,b)$ discussed above, $b-a \neq k$.
This also hints at why the $k$-core length is a good measure of complexity in the Hurwitz--Brill--Noether setting: it ignores the ``redundent'' pairs $(a,b)$ whose existence is implied by multiplication with powers of $f$, which in turn would correspond to redundant equations in the codimension estimate.
\end{rem}

\subsection{Refined Hurwitz--Brill--Noether theory: splitting loci} \label{ssec:hbnRefined}
The main question of this paper, Question \ref{qu:underwaterBN}, has a natural analog in Hurwitz--Brill--Noether theory, which to my knowledge has not been studied in detail. To state it, we use the terminology of splitting type loci, which we briefly summarize now.
Splitting type loci provide the right vocabulary for the ``refined'' form of Hurwitz--Brill--Noether theory that cleanly accounts for the reducibility and non-equidimensionality of $G^r_d(C)$ for $C$ a general $k$-gonal curve.
For a fixed degree-$k$ branched cover $f: C \to \PP^1$ from a genus $g$ smooth curve, and a nondecreasing sequence $\vec{e} = (e_1, \cdots, e_k)$ of integers called the \emph{splitting type}, let $W^{\vec{e}}(f)$ denote the locus of line bundles $\cL$ such that $\displaystyle f_\ast \cL \cong \bigoplus_{i=1}^k \cO_{\PP^1}(e_i)$. Since
\begin{equation}
\deg \cL = \displaystyle g+k-1 + \sum_{i=1}^k e_i \hspace{1cm} \mbox{and} \hspace{1cm} h^0(C,\cL) = \displaystyle  \sum_{i=1}^k \max \{0, e_i + 1\},
\end{equation}
the stratification into splitting type refines the stratification into Brill--Noether loci. The expected codimension of $W^{\vec{e}}(f)$, from Larson's theory of splitting type loci \cite{larsonDegen}, is
$$u(\vec{e}) = h^1(\PP^1, End(f_\ast \cL)) = \sum_{i,j} \max \{ 0, e_i - e_j - 1\}.$$

Informally, the more generic splitting types are those that are more ``balanced.''
In fact, this expected codimension can be described in terms of displacement. To every splitting type $\vec{e}$, we associate a $k$-core partition $\Gamma(\vec{e})$ \cite[Definition 4.7]{larsonLarsonVogt}, and we have \cite[Proposition 5.6]{larsonLarsonVogt}:
\begin{equation}
u(\vec{e}) = \left| \Gamma(\vec{e}) \right|_k.
\end{equation}

\begin{rem}
Remark \ref{rem:whyKCores} suggests a succinct description of $\Gamma(\vec{e})$ in the case that $f: C \to \PP^1$ has a point $p$ of total ramification: for $n \gg 0$, it is the partition obtained from the ramification sequence at $p$ of the complete linear series of $\cL(np)$. Here $n$ must be large enough that $h^1(C, \cL(np)) = 0$, because that means that incrementing $n$ simply adds one more $0$ to the ramification sequence.
\end{rem}

For $f$ a general point in Hurwitz space, this expected codimension $u(\vec{e})$ is correct:
\begin{equation}
\dim W^{\vec{e}}(f) = g - u(\vec{e}).
\end{equation}
This was proved independently in \cite{larsonHBN} and \cite{cpjComponents, cpjMethods}; it is the analog over Hurwitz space of the Brill--Noether theorem. Furthermore, $W^{\vec{e}}(f)$ is irreducible \cite{larsonLarsonVogt}. From this, one can classify the irreducible components of $G^r_d(C)$ for a general $k$-gonal curve $C$, by first classifying all maximal splitting types corresponding to $d$ and $r$; see \cite[Corollary 1.3]{larsonHBN} or \cite[\S 2.2]{cpjComponents}.

\subsection{An analog of Question \ref{qu:underwaterBN}}
This construction can be relativized, to obtain a moduli space $\cW^{\vec{e}}_g \to \cH_{k,g}$, where $\cH_{k,g}$ is the Hurwitz space.
We may now ask the analog of Question \ref{qu:underwaterBN} for Hurwitz space.

\begin{qu}\label{qu:underwaterHBN}
For which $g, \vec{e}$ does $\cW^{\vec{e}}_g \to \cH_{k,g}$ have a component of relative dimension $g - u(\vec{e})$ and generic fiber dimension $\max\{0, g-u(\vec{e})\}$?
\end{qu}

The main theorems of Hurwitz--Brill--Noether theory imply that, in all cases where $g > u(\vec{e})$, there is a \emph{unique} dimensionally proper component. This leaves the non-surjective case $g < u(\vec{e})$. 

I believe it may be productive to generalize the tools of this paper to bear on Question \ref{qu:underwaterHBN}, but there are a number of details that will require care. In particular, I believe that the machinery of threshold genera of skew shapes may be a useful tool, but that it should be necessary to restrict all partitions to be $k$-core partitions.
Hopefully, a subadditivity theorem, analogous to Theorem \ref{thm:subadd}, holds for $k$-core partitions; the proof of such a theorem might be possible by adapting one of the smoothing techniques developed so far \cite{jensenRanganathan, larsonHBN, larsonLarsonVogt}.
Finally, I will remark that elliptic chains are unlikely to be useful in approaching Question \ref{qu:underwaterHBN}. This is because the elliptic chains considered in \cite{larsonLarsonVogt} and elsewhere already have $k$-torsion on each curve in the chain, so they cannot specialize further, yet they already behave like general points in Hurwitz space. Novel ideas are needed.

\section*{Acknowledgements}

Many of the underlying ideas in this paper were originally developed as part of my PhD thesis, and I am grateful to my adviser Joe Harris for his immeasurable encouragement and help. I would like to thank Angelo Lopez and Dave Jensen for thoughtful comments on a draft of this paper, and the anonymous referee for an attentive reading and the suggestion to write Appendix \ref{app:hbn}. I also received helpful advice about earlier versions of this work from a number of people, including Melody Chan, and David Speyer. The preparation of this manuscript was supported by a Miner D. Crary Sabbatical Fellowship from Amherst College.

\ifthenelse{\equal{\debug}{1}} %
{
\section*{To do}
\begin{enumerate}
\item \empty
\end{enumerate}
}{\empty}

%-----------
%-----------

\bibliographystyle{amsalpha}
\bibliography{../template/main}
\end{document}